\documentclass[a4paper,reqno]{amsart}
\usepackage{amssymb}
\usepackage{latexsym}
\usepackage{amsmath}
\usepackage{euscript,url}
\usepackage{graphics,color}
\usepackage[all]{xy}
\usepackage{mathrsfs}
\usepackage[margin=3cm]{geometry}
\usepackage{MnSymbol} 

   \def\sH{{\mathfrak H}}

      \def\sR{{\mathfrak R}}
      
\def\sV{{\mathfrak V}}

\def\ff{{\mathfrak{f}}}

\def\sV{{\mathscr{V}}}

      \def\cU{{\mathcal U}}
\def\cV{{\mathcal V}}   \def\cW{{\mathcal W}}   \def\cX{{\mathcal X}}
   \def\cZ{{\mathcal Z}}

\newcommand{\be}{\begin{equation}}
\newcommand{\ee}{\end{equation}}
\newcommand{\ba}{\begin{eqnarray}}
\newcommand{\ea}{\end{eqnarray}}
\newcommand{\baa}{\begin{eqnarray*}}
\newcommand{\eaa}{\end{eqnarray*}}
\newcommand{\bb}{\begin{thebibliography}}
\newcommand{\eb}{\end{thebibliography}}

\newcounter{my}
\newcommand{\he}%
   {\stepcounter{equation}\setcounter{my}%
   {\value{equation}}\setcounter{equation}0%
   }%
\newcommand{\she}%
   {\setcounter{equation}{\value{my}}%
    }%

\newtheorem{pr}{Proposition}
\newtheorem{cor}{Corollary}

\newtheorem{theorem}{Theorem}[section]

\newtheorem{definition}[theorem]{Definition}
\theoremstyle{definition}

\newtheorem{remark}[theorem]{Remark}

\numberwithin{equation}{section}

\newcommand{\bra}[1]{\langle\,{#1}}
\newcommand{\ket}[1]{\mid{#1}\,\rangle}

\newcommand{\hg}[2]{\,\mbox{}_{#1}F_{ #2}}

\newcommand{\argu}[3]{\left(\!\!\begin{array}{c} #1\\#2\end{array} \!\!\Bigg\vert \,#3\right)}

\title{Meta Algebras and Special Functions: the Racah Case}
\author[N.~Crampé, Q.~Labriet, L.~Morey, S.~Tsujimoto, L.~Vinet, A.~Zhedanov]{Nicolas Crampé, Quentin Labriet, Lucia Morey, \\Satoshi Tsujimoto, Luc Vinet, Alexei Zhedanov }
\date{February 2026}

\begin{document}

\maketitle

\begin{abstract}
Finite families of biorthogonal rational functions and  orthogonal polynomials of Racah-type are studied within a unified algebraic framework based on the meta Racah algebra and its finite-dimensional representations. These functions are identified as overlap coefficients between eigensolutions of generalized and standard eigenvalue problems posited on the representation space. The approach naturally yields their orthogonality relations and bispectral properties.
\end{abstract}

\section{Introduction}

Orthogonal polynomial families belonging to the Askey scheme \cite{Koekoek} possess a dual spectral structure: they obey a three-term recurrence relation and simultaneously satisfy a differential or difference equation, both interpretable as eigenvalue problems in either the variable or degree picture, for suitable linear operators. These two operators generate the so called Askey–Wilson algebra, or one of its limiting or specialized forms \cite{zhedanov1991hidden}. From this perspective, the orthogonal polynomials are realized as overlap coefficients between the eigenbases of the corresponding two algebra generators within a given irreducible representation.   In the case of finite-dimensional representations, this set of operators constitutes a Leonard pair \cite{TERWILLIGER2001}, a structure that is in one-to-one correspondence with the terminating families in the Askey scheme.

This work advances a long-term research program devoted to a unified algebraic treatment of biorthogonal rational functions.  An algebraic interpretation\footnote{Let us point out that an alternative algebraic approach to rational functions was proposed in \cite{GW2025}, showing that (multivariate) rational functions of $q$-Racah type arise as overlap coefficients in representations of $\mathcal{U}_q(\mathfrak{sl}_2)$.} of these rational functions has already been initiated for those of type Hahn \cite{tsujimoto2021algebraic,Vinet-unified} and $q$-Hahn \cite{bussiere2022bispectrality}, and more recently for the Wilson rational functions \cite{CTVZ25}. These developments led to the introduction of the meta ($q$)-Hahn algebra, or meta Wilson algebra. The finite-dimensional representation theory of the meta ($q$)-Hahn algebra has been provided in \cite{tsujimoto2024meta,bernard2024meta}, where algebraic proofs of the relations satisfied by the corresponding rational functions were obtained.


In the present paper, we turn to the Racah case and investigate terminating ${}_4F_3$  hypergeometric rational functions associated with this setting. Our starting point is the abstract meta Racah algebra, and the rational functions are shown to emerge as overlap coefficients between bases related to GEVPs in its representations.

The structure of the paper is as follows. In Section \ref{sec:metaR}, we define the meta Racah algebra and discuss its connections with several related algebras, including the Hahn, meta Hahn, meta 
$q$-Racah, and Racah algebras. Section \ref{sec:representations} presents finite-dimensional bidiagonal representations of the generators of the meta Racah algebra. In Section \ref{sec:eigenbases}, several GEVPs and EVPs of interest are solved using these bidiagonal representations. Section \ref{sec:representations-various-bases} provides the representations of the meta Racah algebra on various bases, and establishes the connection with the Leonard trio structure introduced recently in \cite{Trio2026} by some of us as a generalization of the notion of Leonard pair. The Racah polynomials are explicitly identified in Section \ref{sec:racah-pols} within the overlap coefficients between two EVP bases; their orthogonality relations and bispectral properties are naturally recovered from the properties of the eigenbases and the representations of the meta Racah algebra on these bases. In Section \ref{sec:racah-rational}, rational functions of Racah type are identified within the overlap coefficients between GEVP and EVP bases, and their biorthogonality relations as well as generalized bispectral properties are derived algebraically using the representations of the meta Racah algebra. A differential model of this algebra is given in Section \ref{sec:models} and shown to yield a unit-circle  contour integral representation of the rational functions of Racah type analogous to the one of the polynomials.

\section{The Meta Racah Algebra \label{sec:metaR}}

The starting point for the algebraic interpretation of the biorthogonal rational functions of Racah type is the meta Racah algebra defined below. Throughout this paper, we use the notations 
$[A,B]=AB-BA$ and $\{A,B\}=AB+BA$  for the commutator and the anti-commutator, respectively.
\begin{definition}\label{def:mR}
    The meta Racah algebra $m\sR$ is the associative algebra with unit $I$ and generators $X$, $V$, $Z$
obeying the defining relations
\begin{align}
    [Z,X] &= Z^2 +X\,, \label{mRA1}\\
    [X,V] &= \{V,Z\}+2\zeta X +2\zeta^2 Z + \xi I\,, \label{mRA2}\\
    [V,Z] &= V +2 X +2\zeta Z+ \eta I\,, \label{mRA3}
\end{align}
with $\eta$, $\zeta$ and $\xi$ central elements. 
\end{definition}
 The meta Racah algebra contains a Casimir element given by
\begin{align}\label{eq:Cas}
  C=2ZVZ+\{X,V\}+2\zeta \{X,Z\}+2X^2+2\zeta^2 Z^2+ 2\eta X+V+2\xi Z\,,
\end{align}
which commutes with the generators $ X$, $Z$ and $ V$.
\begin{remark}
Let us define the following generators:
\begin{align}\label{eq:isob}
\overline{Z}=Z-\frac{\zeta}{2}\,,\quad\overline{X}=X+\zeta Z-\frac{\zeta^2}{4}\,,\quad\overline{V}=V,
\end{align}
which satisfy the following commutation relations
\begin{align}
    [\overline{Z},\overline{X}] &= \overline{Z}^2 +\overline{X}, \label{mRA1b}\\
    [\overline{X},\overline{V}] &= \{\overline{V},\overline{Z}\} + \overline{\xi} I, \label{mRA2b}\\
    [\overline{V},\overline{Z}] &= \overline{V} +2 \overline{X}+ \overline{\eta} I, \label{mRA3b}
\end{align}
where $\overline{\xi}=\xi-\eta\zeta$ and $\overline{\eta}=\eta+\frac{1}{2}\zeta^2$. 
The above presentation of the meta Racah algebra is equivalent to that given in Definition~\ref{def:mR} and provides a more convenient framework for establishing connections with other algebras, as illustrated in the following paragraphs. We emphasize, however, that the presentation given in Definition~\ref{def:mR} is the one that is most relevant for the study of special functions.  Although this presentation of the meta Racah algebra is equivalent to that given in Definition~\ref{def:mR}, $\zeta$ disappears from the equations, and for this reason one cannot properly recover all the parameters of the special functions.
\end{remark}
Next we discuss important connections between the meta Racah algebra and some related algebras.
\paragraph{\textbf{Hahn algebra.}}
Let us remark that $\overline{X}$ can be expressed in terms of $\overline{V}$ and $\overline{Z}$ using relation \eqref{mRA3b}. Substituting this expression, one obtains an equivalent set of defining relations for the meta Racah algebra involving only  $\overline{V}$ and $\overline{Z}$:
\begin{align}
\label{eq:metaracah-hahn}
  &  [[\overline{V},\overline{Z}],\overline{V}]=2\{\overline{V},\overline{Z}\} +2\overline \xi I\,,\\
  \label{eq:metaracah-hahn2}
  &  [\overline Z,[\overline V,\overline Z]]=2\overline Z^2-\overline V-\overline \eta I\,.
\end{align}
These relations coincide with the defining relations of the Hahn algebra. Indeed, upon identifying $K_1=-2\overline{Z}$ and $K_2=-\overline{V}$, one recovers the definition of the Hahn algebra given in equation~(1.3) of~\cite{FGVVZ}. Hence, the Hahn algebra through its connection with the meta Racah algebra will lead to interpretation of rational functions of \textit{Racah} type. 
\noindent

\paragraph{\textbf{Meta Hahn algebra.}}
The meta Hahn algebra $m\sH$ was introduced and studied in \cite{tsujimoto2024meta}. We will show that it can be obtained as a contraction of $m\sR$. To this end, we define new generators by
\begin{align}
   & \overline{Z}=\frac{1}{\epsilon} \widetilde{Z}+\frac{1}{2\epsilon}\,,
    \qquad \overline{X}=\frac{1}{\epsilon} \widetilde{X}-\frac{1}{4\epsilon^2}\,,
    \qquad \overline{V}= \widetilde{V}\,,
  \qquad   \overline{\xi}=\frac{1}{\epsilon}\widetilde{\xi}\,,\qquad \overline{\eta}=\frac{1}{\epsilon}\widetilde{\eta}+\frac{1}{2\epsilon^2}\,.
\end{align}
These new generators satisfy the modified relations
\begin{align}
&   [ \widetilde{Z},\widetilde{X}]=\widetilde{Z}^2+\widetilde{Z}+\epsilon \widetilde{X}\,,\\
 &  [ \widetilde{X},\widetilde{V}]=\{\widetilde{V},\widetilde{Z}\} +\widetilde{V}+\widetilde{\xi}I\,,\\
  &  [ \widetilde{V},\widetilde{Z}]=\epsilon\widetilde{V}+2\widetilde{X}+\widetilde{\eta}I\,.
\end{align}
In the limit $\epsilon\to0$, one recovers the defining relations of the meta Hahn algebra. Let us mention that the $q$-deformation of the meta Hahn algebra has been studied in \cite{bernard2024meta}.\\
\noindent

\paragraph{\textbf{Meta $q$-Racah algebra.}}
The meta $q$-Racah algebra $m\sR_q$ has been defined in \cite{CTVZ25} as the natural algebra encoding the bispectral properties of the $q$-Racah rational functions. This algebra is generated by three elements $\cX$, $\cZ$, $\cV$ subject to the following defining relations:
\begin{align}
\label{eq:meta}
 &[\cZ,\cX]_q=  \ff_1\cZ+\ff_2\cX \,,\\
 &[\cX,\cV]_q= \ff_1 \cV+\ff_6 \cZ+ \ff_7 I\,,\\
  &[\cV,\cZ]_q=\ff_2 \cV+ \ff_4\cX+\ff_5 I\,,
\end{align}
where $[A,B]_q=(AB-qBA)/(1-q)$ is the $q$-commutator and the parameters entering in the definition of $m\sR_q$ depends on 4 parameters $a,b,d,e$ as follows
\begin{subequations}\label{eq:f}
\begin{align}
& \ff_1=-de\,,\quad \ff_2= -q/a \,,
\quad \ff_4=(1+q)a\,,\quad \ff_5= ade(1+q)-q(bed+e+d+1) \,,\\ &\ff_6=\frac{b}{a} (1+q)\,,\quad
\ff_7=\frac{qb}{a^2}(1+q)-\frac{q}{a}(bed+eb+bd+1) \;.
\end{align}
\end{subequations}    
Let define new generators by affine transformations
\begin{align}
  &  \cV=(1-q)^2\overline{\cV}-\frac{\ff_1\ff_2\ff_4}{\omega^2}I, \quad \cZ=(1-q)\omega \overline{\cZ}-\omega + \ff_2,\\
   & \cX=(1-q)^2\frac{\ff_1\ff_2}{\omega} \overline{\cX}-(1-q)\frac{\ff_1\ff_2}{\omega}\overline{\cZ}-\frac{\ff_1\ff_2}{\omega} + \ff_1\,,
\end{align}
and put $q=\exp(h)$, $b=\exp(h\overline{b})$, $c=\exp(h\overline{c})$, $d=\exp(h\overline{d})$, $e=\exp(h\overline{e})$, $\omega=-\exp(h((1+\overline{e}+\overline{d})/2-\overline{b}/4))$. We can show that $\overline{\cX}$, $\overline{\cZ}$, $\overline{\cV}$ in the limit $h\to 0$ satisfy the defining relations \eqref{mRA1b}-\eqref{mRA3b} of $m\sR$.\\

\noindent

\paragraph{\textbf{Racah algebra.}} The subalgebra of the meta Racah algebra generated by $X+\rho Z$ (with $\rho$ a free parameter), $V$ and its Casimir element $C$, defined by \eqref{eq:Cas}, is the Racah algebra \cite{GZ88}. Indeed, one can show that these generators satisfied the following relations, which are precisely the defining relations of the Racah algebra:
\begin{align}
  &  [V,[W,V]]=2\{W,V\}+2V^2 +2(\eta+\zeta(\zeta-\rho)) V+2(\rho\xi+\zeta(\zeta\eta-\xi-\eta\rho)) I\,,\\
  & [[W,V],W]=2\{W,V\}+2 W^2+2(\eta+\zeta(\zeta-\rho)) W+(1-\rho^2)V-C+\rho(\xi-\rho\eta)I\,.
\end{align}
The Racah algebra encodes the bispectrality of the eponymous polynomials: we show in the following that the overlap coefficients between the basis diagonalizing $W$ and the one
diagonalizing $V$ are given explicitly in terms of the Racah polynomials.\\

\noindent

\paragraph{\textbf{Borel subalgebra of $\mathfrak{sl}_2$.}}
The subalgebra of the meta Racah algebra generated by $X$ and $Z$ subject only to \eqref{mRA1} is isomorphic to a Borel subalgebra of the Lie algebra $\mathfrak{sl}_2$ (see \cite[Proposition 5.1]{Gaddis}). Indeed,  let define $E=X+Z^2$ and $H=Z$. They satisfy $[H,E]=E$ which is the defining relation of the Borel subalgebra of $\mathfrak{sl}_2$.
 Another way to say it is that it is isomorphic to the enveloping algebra of the non-abelian two-dimensional solvable Lie algebra.

\section{Two-diagonal Representation}
\label{sec:representations}

Let $N$ be a positive integer, and $\sV_N$ be a real vector space of dimension $N+1$. The real scalar product of
two vectors $\ket{v}$ and $\ket{w}$ in  $\sV_N$ will be denoted by $\bra{v}\ket{w}$. We denote the vectors of the so-called standard basis of $\sV_N$ by $\ket{n}$, for $ n=0,1,\ldots,N,$ such
that $\bra{m}\ket{n} = \delta_{m,n}$ for all $m, n = 0, 1, \dots, N$. 

The representations of the Hahn algebra  \eqref{eq:metaracah-hahn}-\eqref{eq:metaracah-hahn2}, in which the generators $\overline{Z}$ and $\overline{V}$ act in a bidiagonal fashion on the standard basis, are known \cite{terwilliger2008}
\begin{align}\label{eq:Zfb}
     \overline{Z}\ket{n} &= (n-\overline{\alpha}) \ket{n} +  \ket{n+1}\,,\qquad \overline{Z}\ket{N} = (N-\overline{\alpha}) \ket{N}\,, \\
    \label{eq:Xfb}
     \overline{V}\ket{n} &= (n-\overline{\beta} - 1)(\overline{\beta}-n )\ket{n}  +n (N+1-n)(n-1-2\overline{\alpha}-\overline{\beta}+N)\ket{n-1}\,,
\end{align}
where $\overline{\alpha},\overline{\beta}$ are two parameters related
to the algebra parameters $\overline{\eta}$ and $\overline{\xi}$ as follows:
\begin{align}
 &\overline{\xi}=(\overline{\beta}+1 )(\overline{\beta} - N)( N-2\overline{\alpha}),\quad \overline{\eta} = \frac12\big((2\overline{\alpha}-N)^2+(\overline{\beta}-N)^2+\overline{\beta}(\overline{\beta} + 2)\big).
\end{align}
This basis is referred to as the split basis in \cite{terwilliger2008}. 
Observe that, by a suitable renormalization of the vectors, one can always choose the coefficient in front $\ket{n+1}$ for the action of $\overline{Z}$. As a consequence of equation \eqref{mRA3b}, the action of $\overline{X}$ on the standard basis is given by 
\begin{align}
  \overline{X}\ket{n} &= -(n-\overline{\alpha})^2 \ket{n} - (n-\overline{\beta}) \ket{n+1},\hspace{0.7cm} \overline{X}\ket{N} = -(N-\overline{\alpha})^2 \ket{N}.
\end{align}
Using isomorphism \eqref{eq:isob}, the actions of $Z$, $V$ and $X$ can be deduced as
   \begin{align}\label{eq:Zf}
     {Z}\ket{n} &= (n-\alpha) \ket{n} +  \ket{n+1},\qquad {Z}\ket{N} = (N-\alpha) \ket{N}\,, \\
    \label{eq:Vf}
     {V}\ket{n} &= (n-\beta-\zeta-1)(\beta+\zeta-n)\ket{n}+n(N+1-n)(n-1-2\alpha-\beta-2\zeta+N)\ket{n-1}\,,\\
      \label{eq:Xf}
     {X}\ket{n} &= -(n-{\alpha})^2 \ket{n} - (n-{\beta}) \ket{n+1},\hspace{0.7cm} X\ket{N} = -(N-{\alpha})^2 \ket{N}.
\end{align} 
where $\alpha=\overline{\alpha}-\frac{\zeta}{2}$ and  $\beta=\overline{\beta}-\zeta$. 
This provides a representation of the defining relations \eqref{mRA1}-\eqref{mRA3} with
\begin{align}
  &  \xi=(\beta+1)(\beta-\zeta-N)(N-2\alpha)+2\alpha\zeta(\alpha+1)\,,\label{eq:xi2}\\
   & \eta=(N-\zeta)(N-2\alpha-\beta-\zeta)+(\beta+\zeta)(\beta+1)+2\alpha^2\,.\label{eq:eta2}
\end{align}
For convenience, we record the action of the transposed generators:
\begin{align}
    Z^{\top}\ket{n} &= (n-\alpha) \ket{n} +  \ket{n-1}\label{eq:ZT-standard}\,, \\
        V^{\top}\ket{n} &= (n-\beta-\zeta-1)(\beta+\zeta-n)\ket{n}+(n+1)(N-n)(n-2\alpha-\beta-2\zeta+N)\ket{n+1}\label{eq:VT-standard}\,,\\
          X^{\top}\ket{n} &= -(n-\alpha)^2 \ket{n} - (n-1-\beta) \ket{n-1}\label{eq:XT-standard}\,.
\end{align}
In what follows, the notion of algebraic Heun operator will be needed.
An algebraic Heun operator associated to a LP $(V,Z)$ is the linear combination
\begin{align}\label{eq:Heun}
    H=h_0 I +h_1 Z +h_2 V +h_3 ZV + h_4 VZ\,,
\end{align}
 for $h_0,h_1,h_2,h_3,h_4\in \mathbb{C}$. It has been demonstrated in \cite{NT} that any endomorphism of a finite dimensional vector space which is tridiagonal in both bases associated to the LP $(X,Y)$ can be written as \eqref{eq:Heun}. The name appears for the first time in \cite{algebraicheun-vinet}, in the study of the band and time limiting problem, as a generalization of the differential Heun operator. It has also been used to simplify the computation of the entanglement entropy \cite{CNV19} and appears in the study of quantum integrable systems \cite{BP19}. We emphasize that the operator $ X$ defined by \eqref{mRA3}  is a particular algebraic Heun operator associated to the Leonard pair $( V, Z)$. More precisely, it is the most general algebraic Heun operator associated to the Leonard pair $(V,Z)$ which acts bidiagonally as $ Z$. This result is deduced from the following proposition.
\begin{pr}
The most general algebraic Heun operator associated to the Leonard pair $( V,  Z)$ acting bidiagonally as $Z$ is \begin{equation}\label{eq:Heunbi}
 H=h_0 I+ h_1  Z-h_4 V+h_4 [ V, Z],
\end{equation}
for $h_0,\, h_1,\, h_4 \in \mathbb{C}$.
\end{pr}
\begin{proof}
For simplicity, let us denote the bidiagonal representations \eqref{eq:Zfb}, \eqref{eq:Xfb} of $ Z$ and $ V$ by
\begin{align}
 Z \ket{n} &= \alpha_n \ket{n} + \ket{n+1}, \\
 V \ket{n} &= \beta_n \ket{n} + \gamma_n \ket{n-1},
\end{align}
and consider the algebraic Heun operator $H$ associated with the Leonard pair $( V,  Z)$ defined in \eqref{eq:Heun} 
The action of $H$ on the standard basis $\ket{n}$ is given by
\begin{align}
 H\ket{n}&=(h_1+h_3\beta_n+h_4\beta_{n+1})\ket{n+1}+(h_0+h_1\alpha_n+h_2\beta_n+h_3\beta_n\alpha_n+h_3\gamma_n+h_4\alpha_n\beta_n+h_4\gamma_{n+1})\ket{n}\nonumber\\
&+\gamma_n(h_2+h_3\alpha_{n-1}+h_4\alpha_n)\ket{n-1}\,.
\end{align}
From this expression, and since $\gamma_n\neq 0$, we see that $H$ acts in a bidiagonal fashion (as $ Z$ does) if and only if
\begin{equation}
h_2+h_3\alpha_{n-1}+h_4\alpha_n =0\,.
\end{equation}
Using $\alpha_n=n- \alpha$, we get $h_3=h_2=-h_4$. Thus, $H$ takes the form \eqref{eq:Heunbi} which concludes the proof of the proposition.
\end{proof}

\section{(Generalized) Eigenbases}
\label{sec:eigenbases}
Finite-dimensional representations of the meta Racah algebra have been obtained in the previous section.
In order to study biorthogonal rational functions some GEVP appears naturally. Hence it is important to consider GEVP bases associated with the elements \( X \) and \( Z \).
We also consider EVP bases associated with the elements \( X + \rho Z \) and \( V \). In the forthcoming sections, the orthogonal polynomials and bispectral rational functions will then  be identified explicitly as overlap coefficients between different (generalized) eigenbases, and their properties are derived using the actions of the generators in these various bases. We provide these bases below and additionally, we include the EVP associated with \( Z \), which will be needed later for the connection with the approach using Leonard trios.

More precisely, the different (generalized) eigenbases useful in this paper are: 
\begin{itemize}
\item The vectors $\ket{d_n}$ and $\ket{d^*_n}$ are the eigenvectors of the GEVP
\begin{align}
  &  X \ket{d_n} = \lambda_n Z \ket{d_n}\,,\\
  &   X^{\top} \ket{d^*_n} = \lambda_n Z^{\top} \ket{d^*_n}\,,\label{eq:GEVPds}
\end{align}
where $\lambda_n =\alpha-n$.
\item 
The vectors $\ket{e_n}$ and $\ket{e^*_n}$  are the eigenvectors of the EVP
\begin{align}
 &   V \ket{e_n} = \mu_n \ket{e_n}\,,\label{eq:Ven}\\
 &   V^{\top} \ket{e_n^*} = \mu_n \ket{e_n^*}\,,
\end{align}
where $\mu_n =(n-\beta-\zeta-1)(\beta+\zeta-n)$.

\item The vectors $\ket{f_n}$ and $\ket{f^*_n}$ are the eigenvectors of the EVP 
\begin{align}
&(X+\rho Z) \ket{f_n} =  (n-\alpha-\rho)(\alpha-n) \ket{f_n}\,,\\
&(X^{\top}+\rho Z^{\top}) \ket{f_n^*} =  (n-\alpha-\rho)(\alpha-n)\ket{f_n^*}\,,
\end{align}
where $\rho$ is a real parameter. 
\item The vectors $\ket{z_n}$ and $\ket{z^*_n}$  are the eigenvectors of the EVP 
\begin{align}
&Z \ket{z_n} = (n-\alpha)\ket{z_n}\,,\\
&Z^{\top} \ket{z^*_n} =(n-\alpha)\ket{z^*_n}\,.
\end{align}
\end{itemize}
The explicit expressions of the eigenvalues in the previous relations are straightforwardly identified from the diagonal part of the actions of the generators on the standard basis. The normalizations of these vectors have been chosen such that the following orthogonality relations hold:
\begin{align}
    & \bra{e_m^*}\ket{e_n}=\bra{e_n}\ket{ e_m^*} =\delta_{n,m}\,,\label{eq:orth-e}\\
    &\bra{ f_m^*}\ket{f_n}=\bra{ f_n}\ket{f_m^*}=\delta_{n,m}\,,\label{eq:orth-f}\\
    &\bra{ z_m^*}\ket{z_n}=\bra{ z_n}\ket{z_m^*}=\delta_{n,m}\,,\label{eq:orth-z}\\
    & \bra{d_m^*}\,|\,Z\ket { d_n}=\bra{ d_m}\,|\,Z^\top\ket{d_n^*} = \delta_{n,m}\,,\label{eq:orth-d}
\end{align}
for $m,n = 0,1,\ldots,N$. From these previous orthogonality relations, one deduces the corresponding
completeness relations
\begin{align}
\sum_{n=0}^N\ket{e_n}\bra{e_n^*}\,\,|&=
\sum_{n=0}^N\ket{f_n}\bra{f_n^*}\,\,|=
\sum_{n=0}^N\ket{z_n}\bra{z_n^*}\,\,|=1\,,\label{eq:completeness}\\
&\sum_{n=0}^N Z\ket{d_n}\bra{d_n^*}\,\,|=1\,.\label{eq:completeness-d}
\end{align}
In the following, we study the overlap coefficients between these different bases. To obtain explicit expressions for these coefficients, we first express the corresponding vectors in the standard basis. The bidiagonal actions derived in the previous section play a crucial role in explicitly constructing these bases and determining the associated (generalized) eigenvalues. The following notation for multiple Pochhammer symbols 
\begin{equation}
(a_1,a_2,\ldots,a_k)_n=(a_1)_n(a_2)_n\ldots(a_k)_n,
\end{equation}
shall be used.

\begin{pr}
\label{pr:basis-in-standard}
The solutions to the GEVPs and EVPs of interest are given as follows in terms
of expansions over the standard basis $\{\ket{\ell}: \ell=0,\ldots,N \}:$
\begin{align}
\label{eq:expression-dn}
   \ket{ d_n} &= \frac{(n-N-\alpha+\beta+1)_{N-n}}{(n-N,\alpha-N)_{N-n}} \sum_{\ell=0}^{N}\frac{(n-N,\alpha-N)_{N-\ell}}{(n-N-\alpha+\beta+1)_{N-\ell}} \ket{\ell},\\
    \ket{d_n^*} &= \frac{(-n+\alpha-\beta)_n}{n!(-\alpha)_{n+1}} \sum_{\ell=0}^{N} (-1)^{\ell} \frac{(-n,-\alpha)_{\ell}}{
    (-n+\alpha-\beta)_{\ell}} \ket{\ell},\label{eq:expression-dnast}\\
   \ket{e_n}  &=\frac{(-N,N-2\alpha-\beta-2\zeta)_n}{(n-2\beta-2\zeta-1)_n} \sum_{\ell=0}^{N} (-1)^{\ell}\frac{(-n,n-2\beta-2\zeta-1)_\ell}{\ell!(-N,N-2\alpha-\beta-2\zeta)_\ell}\ket{\ell},\label{eq:expression-en} \\
   \ket{ e_n^*}  &=\frac{(-N,n+N-2\alpha-\beta-2\zeta)_{N-n}}{(2\beta+2\zeta-N-n+1)_{N-n}}\sum_{\ell=0}^{N} \frac{(n-N,2\beta+2\zeta-N-n+1)_{N-\ell}}{(N-\ell)!(-N,2\alpha+\beta+2\zeta-2N+1)_{N-\ell}}\ket{\ell},\label{eq:expression-en-ast} \\
    \ket{f_n}  &=\frac{(\beta+\rho-N+1)_{N-n}}{(n-N,2\alpha+\rho-N-n)_{N-n}} \sum_{\ell=0}^{N}
   \frac{(n-N,2\alpha+\rho-N-n)_{N-\ell}}{(\beta+\rho-N+1)_{N-\ell}} \ket{\ell},\label{eq:expression-fn}\\
   \ket{f_n^*}&=\frac{(-\beta-\rho)_{n}}{n!(n-2\alpha-\rho)_{n}}\sum_{\ell=0}^{N} (-1)^{\ell}
  \frac{(-n,n-2\alpha-\rho)_{\ell} }{(-\beta-\rho)_{\ell}} \ket{\ell},\label{eq:expression-fn-ast}\\
    \ket{z_n} &= \frac{1}{(n-N)_{N-n}}\sum_{\ell=0}^{N} (n-N)_{N-\ell}\ket{\ell},\label{eq:expression-zn}\\\label{eq:expression-zn-ast}
    \ket{z^*_n}&=\frac{1}{(-n)_n}\sum_{\ell=0}^N(-1)^{\ell+n}(-n)_\ell\ket{\ell}.
    \end{align}
\end{pr}
\begin{proof}
Here is a sketch of the proof. 

(i) Consider to begin, the eigenvectors $\{\ket{e_n}\}_{n=0}^N$ of $V$ in the $N+1$-dimensional vector space $\sV_N$. From $V\ket{e_n}=\mu_n\ket{e_n}$, we have 
\begin{equation}
\bra{\ell}\,\,|(V-\mu_n)\ket{e_n}=0, \qquad (\ell=0,\ldots,N),
\end{equation}
which upon using the action of $V^{\top}$ on the elements $\bra{\ell}\,\,|$ amounts to
\begin{equation}
\label{eq:recurrence-en}
(\mu_\ell-\mu_n)\bra{\ell}\ket{e_n}+ (\ell+1)(N-\ell)(\ell-2\alpha-\beta-2\zeta+N)\bra{\ell+1}\ket{e_n}=0, \quad (\ell=0,\ldots,N),
\end{equation}
where we recall that $\mu_\ell=(\ell-\beta-\zeta-1)(\beta+\zeta-\ell).$ It is easy to solve \eqref{eq:recurrence-en} under the normalization condition $\bra{n}\ket{e_n}=1$ to find 
\begin{equation}
\bra{\ell}\ket{e_n}=(-1)^{n+\ell}\frac{n!(-n,n-2\beta-2\zeta-1)_\ell(-N,N-2\alpha-\beta-2\zeta)_n}{\ell!(-N,N-2\alpha-\beta-2\zeta)_\ell(-n,n-2\beta-2\zeta-1)_n},
\end{equation}
which leads to the expression \eqref{eq:expression-en} of $\ket{e_n}$ after some 
 simplifications.

(ii) The generalized eigenvectors $\{\ket{d_n}\}_{n=0}^N$ satisfy
\begin{equation}
\bra{N-\ell\,\,}|(X-\lambda_nZ)\ket{d_n}=0, \qquad (\ell=0,\ldots,N),
\end{equation}
which leads to 
\begin{equation}
\label{eq:recurrence-dn}
(-\lambda_{N-\ell}^2+\lambda_n\lambda_{N-\ell})\bra{N-\ell}\ket{d_n}+(-N+\ell+1+\beta-\lambda_n)\bra{N-\ell-1}\ket{d_n}=0, 
\end{equation}
where $\lambda_k=\alpha-k$. Solving \eqref{eq:recurrence-dn} under the normalization $\bra{n}\ket{d_n}=1$, yields
\begin{equation}
\bra{\ell}\ket{d_n}=\frac{(n-N,\alpha-N)_{N-\ell}(n-N-\alpha+\beta+1)_{N-n}}{(n-N,\alpha-N)_{N-n}(n-N-\alpha+\beta+1)_{N-\ell}}\,,
\end{equation} 
which proves the expression \eqref{eq:expression-dn} for $\ket{d_n}$.

(iii) Similarly, the eigenvectors $\ket{e_n^*}$ and $\ket{d_n^*}$ of the problems involving the transposed operators $V^{\top}, X^{\top}$ and $Z^{\top}$ are respectively obtained by solving 
\begin{equation}
\bra{N-\ell}\,\,|(V^{\top}-\mu_n)\ket{e_n^*}=0, \quad \bra{\ell}\,\,|(X^{\top}-\lambda_nZ^{\top})\ket{d_n^*}=0\,,
\end{equation}
which entails
\begin{align}
&(\mu_{N-\ell}-\mu_n)\bra{N-\ell}\ket{e_n^*}+(N-\ell)(\ell+1)(2N-\ell-2\alpha-\beta-2\zeta-1)\bra{N-\ell-1}\ket{e_n^*}=0\,,\\
&\lambda_\ell(\lambda_n-\lambda_\ell)\bra{\ell}\ket{d_n^*}+(\beta-\ell-\lambda_n)\bra{\ell+1}\ket{d_n^*}=0\,,
\end{align}
for $\ell=0,\ldots,N.$ 

(iv) One proceeds in the same fashion to obtain the expressions for $\ket{f_n}$, $\ket{f_n^*}$,  $\ket{z_n}$ and $\ket{z_n^*}$.
\end{proof}

\section{Representations of $m\sR$ on various bases}
\label{sec:representations-various-bases}
In this section, the actions of the generators of $m\sR$ on the vectors of the bases obtained in the previous section are provided.  These will prove useful for characterizing the
recurrence and difference properties of the special functions of Racah type. For an operator $O$ acting on
the vector space $\sV_N$, the notation $O^{(b)}_{m,n}$ will refer to its $m, n$ matrix entry in the basis $b \in\{d, d^\ast, e, e^\ast, f, f^\ast, z, z^*\}$. Note moreover that for $b \in \{e, f , z\}$, the resolutions of the identity \eqref{eq:completeness} imply that
\begin{equation}
\label{eq:rel-o}
O^{\top (b^*)}_{m,n}=O_{n,m}^{(b)}\,.
\end{equation}

\subsection{Representations in the $e$ and $e^*$ bases\label{sec:repe}} The basis $e$ is formed of the eigenvectors solving the eigenvalue equation $V\ket{e_n}=\mu_n\ket{e_n}$. As noted previously, the operators $Z$ and $V$ satisfy the Hahn algebra and therefore, $Z$ is tridiagonal in the basis constituted of the eigenvectors of $V$. Indeed, from the expression \eqref{eq:expression-en} of $\ket{e_n}$ and the action of $Z$ on the standard basis, we obtain, for $n=0,1,\dots,N$,
\begin{align}
 Z \ket{e_n} &= Z^{(e)}_{n+1,n} \ket{e_{n+1}} +Z^{(e)}_{n,n} \ket{e_{n}}+Z^{(e)}_{n-1,n} \ket{e_{n-1}},
\label{action:Zone}
\end{align}
where
\begin{align}
&Z^{(e)}_{n+1,n}=1\label{action:Zone:coe1}, \\
&Z^{(e)}_{n,n}=\frac{n(n-2\beta-2\zeta+N-1)(n+2\alpha-\beta-N-1)}{(2n-2\beta-2\zeta-2)(2n-2\beta-2\zeta-1)}\nonumber\\
&-\frac{(n-N)( n-2\beta  - 2\zeta - 1)(  n- 2\alpha-\beta- 2\zeta + N  )}{(2n-2\beta-2\zeta-1)(2n-2\beta-2\zeta)}-\alpha \,,\label{action:Zone:coe2}\\
&Z^{(e)}_{n-1,n}=n \left(N+1-n\right) \label{action:Zone:coe3}\\&\times\frac{\left(n+2\alpha -\beta-N-1\right)  \left(n-2\beta-2\zeta-2 \right) \left(n-2\beta-2\zeta +N-1\right) \left(n-2\alpha -\beta-2\zeta+N-1 \right)}{\left(2n-2\beta-2\zeta -3\right) \left( 2n-2\beta-2\zeta -2 \right)^{2} \left(2n-2\beta -2\zeta -1\right) }.\nonumber
\end{align}
In the previous relations, we have used the conventions that $\ket{e_{N+1}}=\ket{e_{-1}}=0$. 

As noted previously, the operator $X$ is an algebraic Heun operator for the Leonard pair $(V,Z)$. Therefore, $X$ is tridiagonal in the basis constituted of the eigenvectors of $V$. Indeed, from the expression \eqref{eq:expression-en} of $\ket{e_n}$ and the action of $X$ on the standard basis we obtain, for $n=0,1,\dots,N$,
\begin{align}
  X \ket{e_n} &=X^{(e)}_{n+1,n} \ket{e_{n+1}} +X^{(e)}_{n,n} \ket{e_{n}}+X^{(e)}_{n-1,n} \ket{e_{n-1}},
\label{action:Xone}
\end{align}
where
\begin{align}
&X^{(e)}_{n+1,n}= \beta-n\,,
\label{action:Xone:coe1}
\\ 
&X^{(e)}_{n,n}=\frac{n(n-2\beta-2\zeta+N-1)(n+2\alpha-\beta-N-1)(n-\beta-2\zeta-1)}{(2n-2\beta-2\zeta-2)(2n-2\beta-2\zeta-1)}\nonumber\\
&+\frac{(n-N)( n-2\beta  - 2\zeta - 1)(  n- 2\alpha-\beta- 2\zeta + N  )(n-\beta)}{(2n-2\beta-2\zeta-1)(2n-2\beta-2\zeta)}-\alpha^2, \label{action:Xone:coe2}\\
&X^{(e)}_{n-1,n}
=\left(n-\beta-2\zeta-1\right)Z_{n-1,n}^{(e)}.
\label{action:Xone:coe3}
\end{align}
In the previous relations, we have used the conventions that $\ket{e_{N+1}}=\ket{e_{-1}}=0$. 
The tridiagonal actions of the transposed operators $X^\top $ and $Z^\top $ on $\ket{e_n^*}$ are readily obtained from the formulas above.  

\subsection{Representations in the $f$ and $f^*$ bases\label{sec:repf}} 

The $f$ basis is made out of the eigenvectors of  $X + \rho Z$. As shown in Section \ref{sec:metaR}, $X + \rho Z$ and $V$ generate the Racah algebra. Therefore, in the representation basis where $X + \rho Z$ is diagonal, $V$ is tridiagonal. Indeed a straightforward computation gives
\begin{align}
 V \ket{f_n} &= V^{(f)}_{n+1,n} \ket{f_{n+1}} +V^{(f)}_{n,n} \ket{f_{n}}+V^{(f)}_{n-1,n} \ket{f_{n-1}},\label{action:Vonf}
\end{align}
where 
\begin{align}
&V^{(f)}_{n+1,n}=-\frac{\left(n-2 \alpha +\beta +1\right) \left(n-2 \alpha -\rho \right) \left( n+N-2 \alpha -\rho  +1\right) \left(n-N+\beta -\rho+2\zeta  +1 \right) \left(n-\beta -\rho\right)}{\left(2n-2 \alpha -\rho  \right)\left(2n-2 \alpha -\rho +1\right)^{2} \left(2n-2 \alpha -\rho+2\right) }\,, \label{exp:Vp}\\
&V^{(f)}_{n,n}=\frac{n(n-2\alpha+\beta)(n-N+\beta-\rho+2\zeta)(n+N-2\alpha-\rho)}{(2n-2\alpha-\rho-1)(2n-2\alpha-\rho)} \nonumber\\
&\qquad+\frac{(n-N)(n-2\alpha-\rho)(n-\beta-\rho)(n+N-2\alpha-\beta-2\zeta)}{(2n-2\alpha-\rho)(2n-2\alpha-\rho+1)}-(\beta + \zeta + 1)(\beta + \zeta)
, \label{exp:V0}\\
&V^{(f)}_{n-1,n}=-n(n-N-1)(n+N-2\alpha-\beta-2\zeta-1).\label{exp:Vm}
\end{align}
In the previous relations, we have used the conventions that $\ket{f_{N+1}}=\ket{f_{-1}}=0$. Again, one can use the formulas above
and relation \eqref{eq:rel-o} to obtain the tridiagonal action of the transposed operator $V^\top$ on the basis $f^\ast$.

\subsection{Representations in the $d$ and $d^*$ bases\label{ssec:dds}} 
The bases $d$ and $d^*$ are obtained from the GEVP $(X- \lambda_nZ)\ket{d_n}=0$ and its adjoint. The actions of $Z$ and $X$ in the $\ket{d_n}$ basis is bidiagonal and are given explicitly by:
    \begin{align}
  Z \ket{d_n} &=Z^{(d)}_{n+1,n} \ket{d_{n+1}} +Z^{(d)}_{n,n} \ket{d_{n}},\\
  X \ket{d_n} &=X^{(d)}_{n+1,n} \ket{d_{n+1}} +X^{(d)}_{n,n} \ket{d_{n}},
\end{align}
where
\begin{align}
    &Z^{(d)}_{n+1,n}= \dfrac{n-2\alpha+\beta+1}{n-\alpha+1},
\quad  
Z^{(d)}_{n,n}=n-\alpha,
\\
&X^{(d)}_{n+1,n}=- \dfrac{(n-\alpha)(n-2\alpha+\beta+1)}{n-\alpha+1},
\quad X^{(d)}_{n,n}=-(n-\alpha)^2.
\end{align}
The actions of the transposed operators on $\ket{d_n^*}$ are given by:    
\begin{align}
 Z^{\top} \ket{d_n^*} &=Z^{\top(d*)}_{n,n} \ket{d_{n}^*}+Z^{\top(d*)}_{n-1,n} \ket{d_{n-1}^*}, \\
  X^{\top} \ket{d_n^*} &=X^{\top(d*)}_{n,n} \ket{d_{n}^*}+X^{\top(d*)}_{n-1,n} \ket{d_{n-1}^*},
\end{align}
where
\begin{align}
&Z^{\top(d*)}_{n,n}=n-\alpha, \quad  
Z^{\top(d*)}_{n-1,n}=\dfrac{n-2\alpha+\beta}{n-\alpha},\label{action:Zt-d}
\\
&X^{\top(d*)}_{n,n}=-(n-\alpha)^2, \quad  
X^{\top(d*)}_{n-1,n}=-n+2\alpha-\beta.
\end{align}

The operator $V$ is represented in the basis $\ket{d_n}$ by an Hessenberg matrix but, since it is not used in the following, we do not provide its explicit expression. 
Although $V$ does not take a simple form, the operator $VZ$ is tridiagonal in the basis $\ket{d_n}$:
  \begin{align}\label{eq:VZ on dn}
  VZ \ket{d_n} &=(VZ)^{(d)}_{n+1,n} \ket{d_{n+1}}
  +(VZ)^{(d)}_{n,n} \ket{d_{n}}
 +(VZ)^{(d)}_{n-1,n} \ket{d_{n-1}},
\end{align}  
where
\begin{align}
&(VZ)^{(d)}_{n+1,n}=-\frac{(n-2\alpha+\beta+1)(n-\alpha-\zeta)(n-\alpha-\zeta+1)}{n-\alpha+1}, \\
&(VZ)^{(d)}_{n,n}=N(N-\beta-2\alpha-2\zeta)(n-\alpha+\beta+1) \nonumber\\
&+(\beta+\zeta)(\beta+\zeta+1)(n+\alpha)-2n(n-2\alpha-\zeta)(n-\alpha-\zeta) \label{eq:action-VZcoeff2}
,\\
&(VZ)^{(d)}_{n-1,n}=-n(n-N-1)(n-\alpha)(n+N-2\alpha-\beta-2\zeta-1). 
\end{align}
Similarly, the operator $V^\top Z^\top$ is tridiagonal in the basis $\ket{d^*_n}$:
  \begin{align}
  V^{\top}Z^{\top} \ket{d_n^*} &= (V^{\top}Z^{\top})^{(d*)}_{n+1,n} \ket{d^*_{n+1}} +(V^{\top}Z^{\top})^{(d*)}_{n,n} \ket{d^*_{n}}+(V^{\top}Z^{\top})^{(d*)}_{n-1,n} \ket{d^*_{n-1}},
\end{align}  
where
\begin{align}
&(V^{\top}Z^{\top})^{(d*)}_{n+1,n}=-(n-\alpha+1)(n-N)(n+1)(n+N-2\alpha-\beta-2\zeta)\label{action:VZ-coeff1},
\\
&(V^{\top}Z^{\top})^{(d*)}_{n,n}=
(VZ)^{(d)}_{n,n},\label{action:VZ-coeff2}\\
&(V^{\top}Z^{\top})^{(d*)}_{n-1,n}=-\frac{(n-\alpha-\zeta-1)(n-\alpha-\zeta)(n-2\alpha+\beta)}{n-\alpha}\label{action:VZ-coeff3}.
\end{align}

Let us record here the following formula for the action of $Z$ on $\ket{d_n}$ for later use,
\begin{align}
       Z \ket{d_n} =(\alpha-\beta-1)\frac{(n-N-\alpha+\beta+1)_{N-n}}{(n-N,\alpha-N)_{N-n}}\sum_{\ell=0}^{N} \frac{(n-N)_{N-\ell}(\alpha-N)_{N-\ell+1}}{(n-N-\alpha+\beta+1)_{N-\ell+1}} \ket{\ell}.\label{eq:Z-dn}
\end{align}

 \subsection{Link with the Leonard trios}
The goal of this subsection is to indicate how the present work can be cast in the algebraic setting of Leonard trios introduced in \cite{Trio2026}. Before doing so, let us collect the action of $V$ and $X$ on the eigenbasis $\ket{z_n}$ of $Z$
\begin{align}
    V\ket{z_n}&=V^{(z)}_{n+1,n}\ket{z_{n+1}}+V^{(z)}_{n,n}\ket{z_{n}}+V^{(z)}_{n-1,n}\ket{z_{n-1}}\, \label{eq:VtridiagonZ},\\
    X\ket{z_n}&=-V^{(z)}_{n+1,n}\ket{z_{n+1}}-(n-\alpha)^2\ket{z_{n}}
\end{align}
where
\begin{align}
  &V^{(z)}_{n+1,n}=-(n-2\alpha+\beta+1)\,,\qquad V^{(z)}_{n-1,n}=-n(n-N-1)(n+N-2\alpha-\beta-2\zeta-1)\,,\\
 & V^{(z)}_{n,n}=(n-2\alpha+\beta+1)(n-N)+n(n+N-2\alpha-\beta-2\zeta-1)-(N - \beta-\zeta )(N- \beta-\zeta - 1 )\,.
\end{align}
Moreover, let us consider the operator $\widetilde{V}=XZ^{-1}$ which admits the following form on the basis $\ket{z_n}$ 
\begin{equation}\label{eq:TildeTridiadonz}
    \widetilde{V}\ket{z_n}=-(n-\alpha) \ket{z_n}+\frac{n-2\alpha+\beta+1}{n-\alpha}\ket{z_{n+1}}.
\end{equation}

We will now show that the triplet $(V,\widetilde{V},Z)$ is a Leonard trio, and more precisely, a lower reduced Leonard trio. Let first recall their definition for the convenience of the reader (see \cite[Def. 4.1]{Trio2026}).
\begin{definition}\label{def:rLT}
 Let $\mathbb{V}$ be a finite dimensional complex vector space and $\mathrm{End}(\mathbb{V})$ its space of endomorphisms. A lower reduced Leonard trio is an ordered triplet $(V,\widetilde{V},Z)$ of elements of $\mathrm{End}(\mathbb{V})$ that satisfies the following properties:
   \begin{itemize}
       \item[(i)] There exists a basis of $\mathbb{V}$ with respect to which the matrix representing $V$ is diagonal, the one representing $\widetilde{V} Z$ is tridiagonal, and the one representing $Z$ is irreducible tridiagonal.
       \item[(ii)] There exists a basis of $\mathbb{V}$ with respect to which the matrix representing $\widetilde{V}$ is diagonal, the one  representing $ Z V$ is  tridiagonal, and the one representing $Z$ is irreducible  lower bidiagonal.
       \item[(iii)] There exists a basis of $\mathbb{V}$ with respect to which the matrix representing $Z$ is diagonal, the one representing  $\widetilde{V}$ is irreducible lower bidiagonal, and $V$ is irreducible tridiagonal.
   \end{itemize}
   Here, bidiagonal and tridiagonal matrices are called irreducible when there coefficients on the upper and lower diagonals are all non zero. 
\end{definition}

\begin{pr}
The triplet $(V,\widetilde V, Z)$, with $V$ defined in \eqref{eq:Vf}, $Z$ in \eqref{eq:Zf}, and $\widetilde V = XZ^{-1}$ with $X$ given by \eqref{eq:Xf}, is a lower reduced Leonard trio.
\end{pr}
\begin{proof}
    Let us check that the triplet $(V,\widetilde{V}, Z)$ satisfies the properties required by definition \ref{def:rLT}. First, (iii) is already proved since in the basis $\ket{z_n}$, $Z$ is diagonal by definition, $V$ is tridiagonal thanks to \eqref{eq:VtridiagonZ} and $\widetilde{V}$ is lower bidiagonal according to \eqref{eq:TildeTridiadonz}. Similarly, (i) was proved in Subsection \ref{sec:repe}. We are thus left to 
check (ii). An eigenbasis for $\widetilde{V}$ is provided by the vectors $\ket{\tilde{e}_n} =Z\ket{d_n}$ where the associated eigenvalue is $\alpha-n$ from the definition of $\ket{d_n}$. An expression of these vectors in terms of the basis $\ket{n}$ is given by \eqref{eq:Z-dn}. In terms of the basis $\ket{z_n}$ they are given explicitly by
\begin{equation}
    \ket{\tilde{e}_n}=\sum_{\ell=n}^N\frac{(n-2\alpha+\beta+1)_{\ell-n}}{(\ell-n)!(n-\alpha)_{\ell-n}}\ket{z_\ell}.
\end{equation}
One can then check the following action of $Z$ on $\ket{\tilde{e}_n}$
\begin{equation}
Z\ket{\tilde{e}_n}=(n-\alpha)\ket{\tilde{e}_n}+(n-2\alpha+\beta+1)\ket{\tilde{e}_{n+1}}\,.
\end{equation}
Using the action \eqref{eq:VZ on dn} of $VZ$ on $\ket{d_n}$, one finds that the operator $ZV$ acts tridiagonally on $\ket{\tilde{e}_n}$ via 
\begin{align}
ZV\ket{\tilde{e}_n}
=ZVZ\ket{d_n}
=(VZ)_{n+1,n}^{(d)}\ket{\tilde{e}_{n+1}}+(VZ)_{n,n}^{(d)}\ket{\tilde{e}_{n}}+(VZ)_{n-1,n}^{(d)}\ket{\tilde{e}_{n-1}}\,.
\end{align}

\end{proof}


\section{Racah Polynomials}
\label{sec:racah-pols} It is well-known that the Racah polynomials can be interpreted as overlap functions between basis vectors for  representation spaces of the Racah algebra. Here, in spirit of previous works \cite{tsujimoto2024meta}, \cite{ bernard2024meta} on the ($q$-)Hahn case, we explain how the Racah polynomials can be recovered and characterized in a straightforward manner using the bidiagonal representation of the meta Racah algebra. This approach provides a unified framework for the study of both orthogonal polynomials and bispectral rational functions.
Recall that the Racah polynomials  are defined in terms of the generalized hypergeometric functions as 
\begin{align}
\label{eq:racah-pols}
    R_i(x;\hat\alpha,\hat\beta,\hat\gamma,N)={}_4F_3 \left({{-i,\;i+\hat\alpha+\hat\beta+1, \;-x,\;x+\hat\gamma-N}\atop
{\hat\alpha+1,\; \hat\beta+\hat\gamma+1,\;-N}}\;\Bigg\vert \; 1\right)\,,\qquad i=0,\ldots,N.
\end{align}

As noted previously, the matrices $V$ and $X+\rho Z$ satisfy the Racah algebra and, as a consequence \cite{genest2014superintegrability,genest2014racah}, the overlap coefficients between the bases $\ket{e_n}$ and $\ket{f_n}$ diagonalizing respectively $V$ and $X+\rho Z$ are proportional to the Racah polynomials. We denote these overlap coefficients by
\begin{equation}
S_m(n) = \bra{f_n^*}\ket{e_m}, 
\qquad 
\tilde S_m(n) = \bra{f_n}\ket{e_m^*}.
\end{equation}

\subsection{Identification} We start by identifying precisely the Racah polynomials as overlap coeﬃcients between EVP
bases in the representation of the meta Racah algebra.
\begin{pr}
\label{pr:racah-poly-as-transition-coeff}
The functions $S_m(n)=\bra{f_n^*}\ket{e_m}$ and $\tilde S_m(n) = \bra{f_n}\ket{e_m^*}$ are both expressible in terms of Racah polynomials as follows
\begin{align}
 S_m(n)&=\dfrac{(\hat\alpha+1)_n(-N,\hat\beta+\hat\gamma+1)_m}{n!  (n-N+\hat\gamma)_n(m+\hat\alpha+\hat\beta+1)_m}R_m(n;\hat\alpha,\hat\beta,\hat\gamma,N),\label{eq:sm-Racah}\\
\tilde S_m(n&)=\frac{(-1)^N(-N,\hat\gamma-\hat\alpha-N,-\hat\beta-N)_{N-m}(\hat\alpha+1)_m(-\hat\gamma-\hat\beta-n)_n}{(-N-m-\hat\alpha-\hat\beta-1)_{N-m}(n-N,-\hat\gamma-n)_{N-n}(\hat\gamma-\hat\alpha-N,-N-\hat\beta)_n}  R_m(n;\hat\alpha,\hat\beta,\hat\gamma,N), \label{eq:smtilde-Racah}
\end{align}
with
\begin{equation}
\label{eq:parameters-racah}
\hat\alpha=-\beta-\rho-1, \quad \hat\beta=-\beta+\rho-2\zeta-1,\quad \hat\gamma=N-2\alpha-\rho.
\end{equation}
\end{pr} 
\begin{proof}
Given the expressions \eqref{eq:expression-en} and \eqref{eq:expression-fn-ast} for $\ket{e_n}$ and $\ket{f_n^\ast}$ respectively, using $\bra\ell\ket{k}=\delta_{\ell,l}$ and simple transformations, one straightforwardly finds
\begin{align}
      S_m(n)&=\dfrac{(-\beta-\rho)_n(-N,N-2\alpha-\beta-2\zeta)_m}{n! (n-2\alpha-\rho)_n (m-2\beta-2\zeta-1)_m}
    \sum_{\ell=0}^N \dfrac{(-n,-m,n-2\alpha-\rho,m-2\beta-2\zeta-1)_{\ell}}{\ell! (-N,-\beta-\rho,N-2\alpha-\beta-2\zeta)_{\ell}},
\end{align}
which in view of definition \eqref{eq:racah-pols} is readily seen to yield \eqref{eq:sm-Racah} under the identification \eqref{eq:parameters-racah}. The
derivation of formula \eqref{eq:smtilde-Racah} requires a little more eﬀort. Given the expressions of \eqref{eq:expression-en-ast} and \eqref{eq:expression-fn} of $\ket{e_n^\ast}$ and $\ket{f_n}$ respectively, one finds
\begin{align}
    \tilde S_m(n)&=\frac{(-N,N+m-2\alpha-\beta-2\zeta)_{N-m}(\beta+\rho-N+1)_{N-n}}{(2\beta+2\zeta-N-m+1)_{N-m}(n-N,2\alpha+\rho-N-n)_{N-n}}  \nonumber\\
    &\quad \times
 {}_4F_3 \left({{m-N,-N-m+2\beta+2\zeta+1,n-N,-N+2\alpha+\rho-n}\atop
{-N,2\alpha+\beta+2\zeta-2N+1,\beta+\rho-N+1}}\;\Bigg\vert \; 1\right)\label{eq:hypergeometric-prop}.
\end{align}
The hypergeometric function in \eqref{eq:hypergeometric-prop} needs to be transformed to identify the Racah polynomials. The following formula for terminating balanced $\hg{4}{3}$-series  will be used \cite{GR}
\begin{equation}
\label{eq:Whipple}
{}_4F_3 \left({{-n,a,b,c}\atop
{d,e,f}}\;\Bigg\vert \; 1\right)=\frac{(e-a,f-a)_n}{(e,f)_n}{}_4F_3 \left({{-n,a,d-b,d-c}\atop
{d,a+1-n-e,a+1-n-f}}\;\Bigg\vert \; 1\right)\,.
\end{equation}
where $n$ is a non-negative integer and  $1-n+a+b+c=d+e+f$.
Applying twice \eqref{eq:Whipple} with $a=-N-m+2\beta+2\zeta+1,\, b=n-N, c=-N+2\alpha+\rho-n, d=-N, e=2\alpha+\beta+2\zeta-2N+1$ and $f=\beta+\rho-N+1$, then with $a=n-2\alpha-\rho\,, b=m-N,\, c=2\beta+2\zeta-N-m+1,\, d=-N,\, e=\beta-2\alpha+1$ and $f=\beta-\rho+2\zeta-N+1$, one gets 
\begin{align}
&{}_4F_3 \left({{m-N,-N-m+2\beta+1,n-N,-N+2\alpha+\rho-n}\atop
{-N,2\alpha+\beta+1-2N,-N+\beta+\rho+1}}\;\Bigg\vert \; 1\right)\nonumber\\
&=\frac{(m+2\alpha-\beta-N,\rho+m-\beta-2\zeta)_{N-m}}{(2\alpha+\beta+2\zeta-2N+1,\beta+\rho-N+1)_{N-m}}\frac{(2\alpha+2\zeta+\beta-N-n+1,\beta+\rho-n+1)_{n}}{(\beta-2\alpha+1,\beta-\rho+2\zeta-N+1)_{n}}\nonumber\\
&\times {}_4F_3 \left({{-m, m-2\beta-2\zeta-1,-n,n-2\alpha-\rho}\atop
{-N,-\beta-\rho,N-2\alpha-2\zeta-\beta}}\;\Bigg\vert \; 1\right)\,.
\end{align}
Combining these results, one finds
\begin{align}
    \tilde S_m(n)&=\frac{(-1)^{N+m}(-N,m+2\alpha-\beta-N,\rho+m-\beta-2\zeta)_{N-m}(\beta+\rho-N+1)_{N-n}}{(2\beta+2\zeta-N-m+1,\beta+\rho-N+1)_{N-m}(n-N,2\alpha+\rho-N-n)_{N-n}}\\
&\times \frac{(2\alpha+2\zeta+\beta-N-n+1,\beta+\rho-n+1)_{n}}{(\beta-2\alpha+1,\beta-\rho+2\zeta-N+1)_{n}} {}_4F_3 \left({{-m, m-2\beta-2\zeta-1,-n,n-2\alpha-\rho}\atop
{-N,-\beta-\rho,N-2\alpha-2\zeta-\beta}}\;\Bigg\vert \; 1\right)\,.\nonumber
\end{align}
Some simplifications can now be performed. One notes that 
\begin{align}
\frac{(\beta+\rho-N+1)_{N-n}(\beta+\rho-n+1)_n}{(\beta+\rho-N+1)_{N-m}}&=(-1)^m(-\beta-\rho)_m\,,
\end{align}
and, from the identity $(-m-a+1)_m=(-1)^m(a)_m$, that 
\begin{align}
(m+2\alpha-\beta-N)_{N-m}=(-1)^{N-m}(\beta-2\alpha+1)_{N-m}\,,\\
(\rho+m-\beta-2\zeta)_{N-m}=(-1)^{N-m}(\beta+2\zeta-\rho-N+1)_{N-m}\,.
\end{align}
Putting all this together, 
\begin{align}
    \tilde S_m(n)&=\frac{(-1)^{N}(-N,\beta-2\alpha+1,\beta+2\zeta-\rho-N+1)_{N-m}(-\beta-\rho)_m}{(2\beta+2\zeta-N-m+1)_{N-m}(n-N,2\alpha+\rho-N-n)_{N-n}}\nonumber\\
&\times \frac{(2\alpha+2\zeta+\beta-N-n+1)_{n}}{(\beta-2\alpha+1,\beta-\rho+2\zeta-N+1)_{n}} {}_4F_3 \left({{-m, m-2\beta-2\zeta-1,-n,n-2\alpha-\rho}\atop
{-N,-\beta-\rho,N-2\alpha-2\zeta-\beta}}\;\Bigg\vert \; 1\right)\,.
\end{align}
which in view of definition \eqref{eq:racah-pols} is readily seen to yield \eqref{eq:smtilde-Racah} under the identification \eqref{eq:parameters-racah}.
\end{proof}

\subsection{Orthogonality relation.} From the orthogonality relations \eqref{eq:orth-e} and the resolutions of the identity
\eqref{eq:completeness}, the overlap functions
\( S_m(n) \) and \( \tilde S_m(n) \) satisfy an orthogonality relation, reflecting
their orthogonality with respect to the variable $m$:
\begin{align}
\label{eq:orth-S}
    &\sum_{n=0}^N \tilde S_{k}(n) S_{m}(n) = \delta_{k,m}\,.
\end{align}
This property
allows one to recover the orthogonality relation of the Racah polynomials and their
normalization constants. Indeed, substituting the expressions for $S_m(n)$ and $\tilde S_m(n)$ in \eqref{eq:orth-S}
one finds
\begin{align}
    &\sum_{n=0}^N \mathcal{W}_n R_{k}(n,\hat \alpha,\hat \beta,\hat \gamma,N) R_{m}(n,\hat \alpha,\hat \beta,\hat \gamma,N) = \mathcal{N}_m\delta_{k,m},
\end{align}
where 
\begin{align}
   & \mathcal{W}_n=\frac{(-\hat \beta-\hat \gamma-n,\hat \alpha+1)_n}{n!(n-N,-\hat \gamma -n)_{N-n}(\hat \gamma-\hat \alpha-N,-N-\beta,n-N+\hat \gamma)_n},\\
    & \mathcal{N}_m=\frac{(-1)^N(-N-k-\hat \alpha-\hat\beta-1)_{N-m}(m+\hat\alpha+\hat\beta+1)_m}{(-N,\hat\gamma-\hat\alpha-N,-\hat\beta-N)_{N-m}(\hat\alpha+1,-N,\hat\beta+\hat\gamma+1)_m}.
\end{align}
The algebraic approach described here allows us to recover the weight and the normalization for the orthogonality relation of the Racah polynomials. As is classical, a restriction on the parameters is needed for the weight and the norm to be positive \cite{Koekoek}.

\subsection{Recurrence relation and  difference equation.}
We now show explicitly how the bispectral properties of the Racah polynomials, that is, the recurrence relation and diﬀerence equation, follow from the actions of the operators $V$ , $X+ \rho Z$ and their
transposes on the EVP bases. Notice that the diﬀerence equation of the Racah
polynomials can be obtained straightforwardly from their recurrence relation by using their duality property. This result is quite standard \cite{Koekoek}. We rederive it in the framework of the meta Racah algebra representation theory for the sake of illustrating how the meta Racah framework provides a unified treatment of the OPs and BRFs. 

The representation of the operator $V$ in the basis $\ket{e_m}$ and $\ket{f_n}$ computed in Section \ref{sec:repf} allow us to recover the recurrence relation of the Racah polynomials. Indeed, let us compute
\begin{align}
    \bra{f^*_n}\,\,|\,V\ket{e_m}&=\mu_m S_m(n)\\
    &=V^{(f)}_{n,n-1}S_m(n-1)+V^{(f)}_{n,n}S_m(n)+V^{(f)}_{n,n+1}S_m(n+1)\,,
\end{align}
where $V^{(f)}$ are given by \eqref{exp:Vp}-\eqref{exp:Vm}. With the expression of $S_m(n)$ in terms of the Racah polynomials given in proposition \ref{pr:racah-poly-as-transition-coeff}, we recover the recurrence relation of the Racah polynomials.

The difference equation is obtained similarly starting from 
\begin{align}
    \bra{f^*_n}\,\,|\,X+\rho Z \ket{e_m},
\end{align}
and using the representations of $X$ and $Z$ in the basis $\ket{e_n}$ computed in Section \ref{sec:repe}.

\section{Racah Rational Functions}
\label{sec:racah-rational}
This section will provide an algebraic interpretation of the following rational functions 
\begin{align}
& \cU_m(n)=\cU_m(n;a,b,c,N)={}_4F_3 \left({{-m,-n,-a,m-2b-2c-1}\atop
{-N,a-b-n,N-2a-b-2c}}\;\Bigg\vert \; 1\right),\label{eq:rational-cU}
\\
&\widetilde\cU_m(n)=\widetilde\cU_m(n;a,b,c,N)=\cU_{m}(N-n;N-a-1,b+2c-2,2-c,N)\label{eq:rationa-cV}.
\end{align}
\begin{remark}
    In \cite{tsujimoto2024meta} the authors introduced the rational functions of Hahn type. These can be expressed as ${}_3F_2$ hypergeometric functions and can be obtained (up to a normalization) through the following limit
    \begin{equation}
        \lim_{t\to \infty}\mathcal{U}_m(n;t,t-a,\frac{1}{2}(N-1-b+2a)-t)={}_3F_2 \left({{-m,-n,m+b-N}\atop
{-N,a-n}}\;\Bigg\vert \; 1\right). 
    \end{equation}
\end{remark}

\subsection{Representation theoretic interpretation} In the case of BRFs of Racah type, the relevant scalar products to consider are  $U_m(n)=\bra{e_m}\ket{d_n^*}$ and $\widetilde U_m(n) = \bra{e_m^*}\,\,|Z\ket{d_n}$, that is, overlap functions between EVP and GEVP bases in the representation
of $m\sR$.
\begin{pr}
The functions $U_m(n)=\bra{e_m}\ket{d_n^*}$ and $\widetilde U_m(n)= \bra{e_m^*}\,\,|Z\ket{d_n}$ are respectively given as follows in terms of the rational Racah functions $\mathcal{U}_m$ and $\widetilde{\mathcal{U}}_m$:
\begin{align}
U_m(n)&=\frac{(\alpha-\beta-n)_n(-N,N-2\alpha-\beta-2\zeta)_m}{n!(-\alpha)_{n+1}(m-2\beta-2\zeta-1)_m}\cU_m(n;\alpha,\beta,\zeta,N)\,,\label{eq:Un}\\
\widetilde U_m(n) 
    &= \dfrac{ (n-\alpha) (-N+\alpha+\beta+2\zeta,-N+2\alpha-\beta)_{N-n}(m+1,-2N+2\alpha+\beta+2\zeta+1)_{N-m}}{ (n-N, n-\alpha,-2N+2\alpha+\beta+2\zeta+1)_{N-n}(-N-m+2\beta+2\zeta+1)_{N-m}} \, \\&\ \ \times\widetilde\cU_m(n;\alpha,\beta,\zeta,N)\,\nonumber.\label{eq:tildeUn}
\end{align}
\end{pr}
\begin{proof}
The identification of $\mathcal{U}_m(n)$ in $U_m(n)$ is readily achieved by taking the scalar product of the vectors $\ket{e_m}$ and $\ket{d_n^*}$ respectively given in \eqref{eq:expression-en} and \eqref{eq:expression-dnast}. This yields
\begin{align}
U_m(n)= \frac{(-n+\alpha-\beta)_n(-N,N-2\alpha-\beta-2\zeta)_m}{n!(-\alpha)_{n+1}(m-2\beta-2\zeta-1)_m}\sum_{\ell=0}^{N} \frac{(-m,m-2\beta-2\zeta-1,-n,-\alpha)_\ell}{\ell!(-N,N-2\alpha-\beta-2\zeta,-n+\alpha-\beta)_\ell} \,,
\end{align}
from where one gets formula \eqref{eq:Un} using \eqref{eq:rational-cU}. 

Obtaining $\widetilde{U}_m(n)$ requires more algebraic transformations. From formulas \eqref{eq:expression-en-ast} and \eqref{eq:Z-dn} for $\ket{e_n^*}$ and $Z\ket{d_n}$ and the identities
\begin{align}
(a)_{\ell+1}&=a(a+1)_\ell\,,\\
(-N-a)_{N-n}&=(-1)^n(n+a+1)_{N-n}\,,\\
(n-N-a)(a)_{N-n}&=(-1)^{n+1}a(n-N-a)_{N-n}\,,
\end{align} 
one finds 
\begin{align}
\label{eq:Un-1stexp}
    \widetilde U_m(n) 
  &
    \dfrac{ (n-\alpha)(\alpha-\beta-1)_{N-n} (m+1,-2N+2\alpha+\beta+2\zeta+1)_{N-m}}{ ( n-N,n-\alpha)_{N-n}(-N-m+2\beta+2\zeta+1)_{N-m}}\nonumber\\
  & \quad \times {}_4F_3 \left({{n-N,\alpha+1-N,m-N,-N-m+2\beta+2\zeta+1}\atop
{-N,2\alpha+\beta+2\zeta-2N+1,n-N-\alpha+\beta+2}}\;\Bigg\vert \; 1\right)\,.
\end{align}
Applying \eqref{eq:Whipple} with $a=\alpha+1-N\,, b=m-N\,, c=-N-m+2\beta+2\zeta+1\,, d=2\alpha+\beta+2\zeta-2N+1$ and $f=n-N-\alpha+\beta+2$, \eqref{eq:Un-1stexp} becomes
\begin{align}
&   \dfrac{ (n-\alpha)(\alpha-\beta-1)_{N-n} (m+1,2\alpha+\beta+2\zeta+1-2N)_{N-m}(\alpha+\beta+2\zeta-N,n-2\alpha+\beta+1)_{N-n}}{ ( n-N,n-\alpha)_{N-n}(-N-m+2\beta+2\zeta+1)_{N-m}(2\alpha+\beta+2\zeta+1-2N,n-N-\alpha+\beta+2)_{N-n}}\nonumber\\
& \times {}_4F_3 \left({{n-N,\alpha+1-N,-m,m-2\beta-2\zeta-1}\atop
{-N,n-\alpha-\beta-2\zeta+1,2\alpha-\beta-N}}\;\Bigg\vert \; 1\right)\,.
\end{align}
After some simplifications and using \eqref{eq:rationa-cV}, we get 
\begin{multline}
\widetilde{U}_m(n)=\\\dfrac{ (n-\alpha) (m+1,-2N+2\alpha+\beta+2\zeta+1)_{N-m}(-N+\alpha+\beta+2\zeta,-N+2\alpha-\beta)_{N-n}}{ (n-N, n-\alpha,-2N+2\alpha+\beta+2\zeta+1)_{N-n}(-N-m+2\beta+2\zeta+1)_{N-m}} \, \widetilde \cU_m(n;\alpha,\beta,\zeta,N)\,.
\end{multline}
This concludes the proof of the proposition.
\end{proof}

As a first application of this representation theoretic approach we offer another expression for the function $\mathcal{U}_m(n)$ in terms of the dual Hahn polynomials.
This natural in the context of Leonard trios \cite{Trio2026}. The link with the dual Hahn polynomials is a direct consequence of the Hahn algebra appearing as the algebra generated by $V$ and $Z$ as we will see in the proof. The dual Hahn polynomials are defined in terms of the generalized hypergeometric functions as
\begin{align}
    R^{(dH)}_i(x;\boldsymbol\rho)={}_3F_2 \left({{-i, \;-x,\;x+\alpha+\beta+1}\atop
{\alpha+1,\;-N}}\;\Bigg\vert \; 1\right)\,,
\end{align}
where the parameter $\boldsymbol\rho=(\alpha,\beta,N).$
\begin{cor}
\label{cor:rationa-dualhahn}The function $\mathcal{U}_m(n;\alpha,\beta,\zeta,N)$ admits the following expansion in terms of the dual Hahn polynomials
\begin{equation}
    \mathcal{U}_m(n;\alpha,\beta,\zeta,N)=\frac{n!}{(\alpha-\beta-n)_n}\sum_{k=0}^n\frac{(-\alpha)_{k}(2\alpha-\beta-n)_{n-k}}{(n-k)!k!}R_k^{(dH)}(m;\boldsymbol\rho),
\end{equation}
where $R_k^{(dH)}(m;\rho)$ denotes the dual Hahn polynomials with parameters $\boldsymbol\rho=(N-2\alpha-\beta-2\zeta-1, 2\alpha-\beta-N-1,N)$.
\end{cor}
 \begin{proof}
     Using the resolution of the identity \eqref{eq:completeness}, the transition coefficient $U_m(n)=\langle e_m\ket{d_n^*}$ can be rewritten as
     \begin{equation}
         U_m(n)=\sum_{k=0}^N \langle e_m\ket{z_k^*}\langle z_k\ket{d_n^*}.
     \end{equation}
     Using the explicit expressions \eqref{eq:expression-en} and \eqref{eq:expression-zn-ast} for the different vectors appearing in the RHS one finds
     \begin{equation}
     \label{eq:em-zk-dualhahn}
         \langle e_m\ket{z_k^*}=\frac{(-N,N-2\alpha-\beta-2\zeta)_m }{k!(m-2\beta-2\zeta-1)_m }R^{(dH)}_k(m;\boldsymbol\rho).
    \end{equation}
    Using \eqref{eq:expression-dnast} and \eqref{eq:expression-en}, one finds that $\langle z_k\ket{d_n^*}=0$  when $k>n$
    and otherwise one has
    \begin{equation}
   \langle z_k\ket{d_n^*} = (-1)^k\frac{(\alpha-\beta-n)_n(-\alpha)_k(-n)_k}{n!(-\alpha)_{n+1}(\alpha-\beta-n)_k}{}_2F_1 \left({{-n+k,-\alpha+k}\atop
{-n+k+\alpha-\beta}}\;\Bigg\vert \; 1\right)=\frac{(-\alpha)_k(2\alpha-\beta-n)_{n-k}}{(-\alpha)_{n+1}(n-k)!},
     \end{equation}
where we used Chu--Vandermonde's identity for the last equality. The result is then obtained from  \eqref{eq:Un}.
 \end{proof}

\subsection{Biorthogonality} A key feature of the rational functions of Racah type is their biorthogonality. In the following, we derive these biorthogonality relations from the perspective of the meta Racah algebra.
\begin{pr}
The rational functions $\cU_m(n;\alpha,\beta,\zeta,N)$ and $\widetilde \cU_m(j;\alpha,\beta,\zeta,N)$ satisfy the following biorthogonality relations:
\begin{align}
&\sum_{j=0}^N  
\cW(j)\,\widetilde\cU_m(j)\,
\cU_n(j) = h_n \delta_{n,m},\label{eq:orth-rational1}\\
&\sum_{j=0}^N  
\cW^*(j)\,\widetilde\cU_j(m)
\,\cU_j(n) = h_n^*\delta_{n,m}\label{eq:orth-rational2},
\end{align}
where
\begin{align}
& h_n = \frac{n! (N-2\beta-2\zeta)_{n}}{(-N)_n (2n-1-2\beta-2\zeta) (-2\beta-2\zeta)_{n-1}},
\\
& \cW(j) =
\frac{(-N,1-\alpha+\beta,N-2\alpha-\beta-2\zeta)_j
(2\alpha-\beta-N,\alpha+\beta-N+2\zeta)_{N-j} }{j!(-\alpha,-2\beta-2\zeta)_{N}},\\
&h_n^*=\frac{(1,1-2\alpha+\beta,1-\alpha-\beta-2\zeta)_n}{(-N,1-\alpha+\beta,N-2\alpha-\beta-2\zeta)_n},\\
& \cW^*(j) = \frac{(-N)_j}{j!}\frac{(2j-1-2\beta-2\zeta)}{(j-1-2\beta-2\zeta)_{N+1}}\frac{(1-2\alpha+\beta,1-\alpha-\beta-2\zeta)_N}{(-\alpha)_N}.
 \end{align}
\end{pr}
\begin{proof}
As a result of \eqref{eq:orth-e}, \eqref{eq:orth-d}, \eqref{eq:completeness},  \eqref{eq:completeness-d}, the orthogonality relations between $U_{k}(n)$ and $ \widetilde{U}_{m}(n)$ read:
\begin{align}
    &\sum_{n=0}^N \widetilde U_{k}(n) U_{m}(n) = \delta_{k,m}, \\
    &\sum_{m=0}^N \widetilde U_{m}(k) U_m(n) = \delta_{k,n}.
\end{align}
Hence substituting the expressions \eqref{eq:Un} and \eqref{eq:tildeUn} for $U_m(n)$ and $\widetilde U_m(n)$ in terms of $\cU_m(n)$ and $\widetilde \cU_m(n)$ gives the formulas recorded in the above proposition. Note that the weight functions $\cW$ and $\cW^*$ are determined by requesting the normalization $h_0=h_0^*=1$, respectively. For presentation
convenience in obtaining \eqref{eq:orth-rational1}, both sides have been multiplied by $(1,-2N+2\alpha+2\zeta+\beta+1)_N/(-N+2\beta+2\zeta+1)_N$. Similarly, to get \eqref{eq:orth-rational2} both sides have been multiplied by $(-N+\alpha+\beta+2\zeta,-N+2\alpha-\beta)_N/(-N,-\alpha,-2N+2\alpha+2\zeta+\beta+1)_N.$ 
\end{proof}

\subsection{Bispectrality of the Racah rational functions}

The rational functions $\mathcal{U}_m(n)$ satisfy generalized bispectral properties. These are now obtained with the help of the actions of the generators of $m\sR$ on EVP and GEVP bases. We note that the generalized bispectral properties of $\widetilde{\mathcal{U}}_m(n)$ can be deduced from those of $\mathcal{U}_m(n)$  with the help of \eqref{eq:rationa-cV}.
\subsubsection{Recurrence relation}
\begin{pr}
The rational function $\mathcal{U}_m$ of Racah type verifies the following GEVP equation 
\begin{align}
&n\big(\mathcal{A}_m\mathcal{U}_{m+1}(n)-(\mathcal{A}_m+\mathcal{C}_m+\alpha)\mathcal{U}_m(n)+\mathcal{C}_m\mathcal{U}_{m-1}(n)\big)=\label{eq:GEVPrecu}\\
&(m+\alpha-\beta)\mathcal{A}_m\mathcal{U}_{m+1}(n)-\big((m+\alpha-\beta)\mathcal{A}_m-(m-\alpha-\beta -2\zeta-1)\mathcal{C}_m\big)\mathcal{U}_m(n)\nonumber\\&-(m-\alpha-\beta-2\zeta-1)\mathcal{C}_m\mathcal{U}_{m-1}(n)\nonumber\,,
\end{align}
where
\begin{align}
&\mathcal{A}_m=\frac{(m-N)(m+N-2\alpha-\beta-2\zeta)(m-2\beta-2\zeta-1)}{(2m-2\beta-2\zeta-1)(2m-2\beta-2\zeta)}\label{A-recurrence}\,,\\
&\mathcal{C}_m=-\frac{m(m+2\alpha-\beta-N-1)(m-2\beta-2\zeta+N-1)}{(2m-2\beta-2\zeta-2)(2m-2\beta-2\zeta-1)}\,.\label{C-recurrence}
\end{align}
\end{pr}
\begin{proof}
Recalling that $\ket{d_n^*}$ are solutions of the GEVP \eqref{eq:GEVPds}, one gets:
\begin{equation}
    \bra{e_m}\,\,|X^{\top} - \lambda_n Z^{\top}\ket{ d_n^*}=0\,.
\end{equation}
The elements $X$ and $Z$ act tridiagonally on the basis $\ket{e_n}$ (see subsection \ref{sec:repe}) which imply
\begin{align}
\label{eq:recurrence-U-1}
&X^{(e)}_{m+1,m}U_{m+1}(n) +X^{(e)}_{m,m} U_{m}(n)+X^{(e)}_{m-1,m} U_{m-1}(n)
\nonumber\\
&=(\alpha-n)\left(Z^{(e)}_{m+1,m}U_{m+1}(n) +Z^{(e)}_{m,m} U_{m}(n)+Z^{(e)}_{m-1,m} U_{m-1}(n)\right)\,.
\end{align}
Inserting into \eqref{eq:recurrence-U-1} the expression \eqref{eq:Un} of $U_m(n)$ in terms of $\mathcal{U}_m(n)$ and using \eqref{action:Xone:coe1} and \eqref{action:Xone:coe3}, one finds
\begin{align}
&(\beta-\alpha-m+n)Z_{m+1,m}^{(e)}\frac{(-N,N-2\alpha-\beta-2\zeta)_{m+1}}{(m-2\beta-2\zeta)_{m+1}}\mathcal{U}_{m+1}(n)\nonumber\\
&+(X_{m,m}^{(e)}-(\alpha-n)Z_{m,m}^{(e)})\frac{(-N,N-2\alpha-\beta-2\zeta)_{m}}{(m-2\beta-2\zeta-1)_{m}}\mathcal{U}_m(n)\nonumber\\
&+(m+n-\alpha-\beta-2\zeta-1)Z_{m-1,m}^{(e)}\frac{(-N,N-2\alpha-\beta-2\zeta)_{m-1}}{(m-2\beta-2\zeta-2)_{m-1}}\mathcal{U}_{m-1}(n)=0\,.
\end{align}
Multiplying the previous equation by $\frac{(m-2\beta-2\zeta-1)_m}{(-N,N-2\alpha-2\zeta-\beta)_m}$, calling upon the expressions \eqref{action:Zone:coe1} and \eqref{action:Zone:coe3}, and noticing that from \eqref{action:Zone:coe2} and \eqref{action:Xone:coe2} one gets 
\begin{align}
   &Z_{m,m}^{(e)}=-\mathcal{A}_m-\mathcal{C}_m-\alpha\,,\\
    &X_{m,m}^{(e)}=(m-\beta)\mathcal{A}_m-(m-\beta-2\zeta-1)\mathcal{C}_m-\alpha^2\,,
\end{align}
the previous relation becomes the GEVP \eqref{eq:GEVPrecu}.
\end{proof}

\subsubsection{Difference equation}
\begin{pr}
    The rational function $\mathcal{U}_m$ of Racah type obeys the following difference equation
\begin{align}
\label{eq:difference-U}
&\mathcal{B}_n\mathcal{U}_m(n+1)-(\mathcal{B}_n+\mathcal{D}_n)\mathcal{U}_m(n)+\mathcal{D}_n\mathcal{U}_m(n-1)\nonumber\\
&=m(2\beta+2\zeta+1-m)\left((n-\alpha)\mathcal{U}_m(n)- \frac{n(n-2\alpha+\beta)}{n-\alpha+\beta}\mathcal{U}_m(n-1)\right)\,,
\end{align}
where 
\begin{align}
    \mathcal{B}_n=(n-N)(n+N-2\alpha-\beta-2\zeta)(n-\alpha+\beta+1)\,,
    \mathcal{D}_n=n(n-2\alpha+\beta)(n-\alpha-\beta-2\zeta-1).
\end{align}
\end{pr}
\begin{proof}
The EVP \eqref{eq:Ven} $V\ket{e_m}=\mu_m\ket{e_m}$ leads to 
\begin{equation}
    \bra{e_m}\,\,|(V^{\top} - \mu_m) Z^{\top}\ket{  d_n^*}=0\,.
\end{equation}
From fact that $V^{\top}Z^\top$ and $Z^{\top}$ act tridiagonally on  $\ket{d_n^*}$ (see subsection \ref{ssec:dds}), one finds
\begin{align}
\label{eq:difference-eq-U-1}
&
(V^{\top}Z^{\top})^{(d*)}_{n+1,n} U_{m}(n+1)
+(V^{\top}Z^{\top})^{(d*)}_{n,n}U_{m}(n)
+(V^{\top}Z^{\top})^{(d*)}_{n-1,n} U_m(n-1)\nonumber\\
& =\mu_m\left(Z_{n,n}^{\top(d^*)}U_m(n)+Z_{n-1,n}^{\top(d^*)}U_m(n-1)\right)\,.
\end{align}
By substituting into \eqref{eq:difference-eq-U-1} the expression \eqref{eq:Un} of $U_m(n)$ in terms of $\mathcal{U}_m(n)$, using \eqref{action:Zt-d}, \eqref{action:VZ-coeff1}, \eqref{action:VZ-coeff2}, and \eqref{action:VZ-coeff3}, and multiplying the resulting expression by $\frac{n!(-\alpha)_{n+1}}{(\alpha-\beta-n)_n}$, one proves the formula given in the proposition.
\end{proof}

\subsection{Contiguity relations}
We finally provide contiguity relations for the functions $\mathcal{U}_m(n)$ using again the representation theory of $m\sR$.
\begin{pr}
The rational functions $\mathcal{U}_m$ satisfy the following contiguity relations
\begin{align}
\mathcal{U}_m(n;\alpha-1,\beta-2,\zeta+2,N)&=\frac{(n-\alpha)(n-\alpha+\beta)}{\alpha (\alpha-\beta)}\mathcal{U}_m(n;\alpha,\beta,\zeta,N)\nonumber\\&+\frac{n(n-2\alpha+\beta)}{\alpha (\beta-\alpha)}\mathcal{U}_m(n-1;\alpha,\beta,\zeta,N)\,.
\end{align}
\end{pr}
\begin{proof}
The defining relation \eqref{mRA1} of the meta Racah algebra implies that
\begin{align}
   (X^\top Z^\top-Z^\top X^\top-(Z^\top)^2-X^\top)\ket{d_n^*} = 0\,,
\end{align}
which, with GEVP \eqref{eq:GEVPds}, leads to
\begin{align}
    \big(X^\top-2Z^\top-I-(\lambda_n-1) (Z^\top+I)\big)Z^\top \ket{d_n^*} = 0.
\end{align}
Using definitions \eqref{eq:Zf} and \eqref{eq:Xf} of $Z$ and $X$ respectevely, one gets
\begin{align}
 &X-2Z-I = X\big|_{\alpha\to\alpha-1,\beta\to\beta-2,\zeta\to\zeta+2}\,,\quad
 Z+I = Z\big|_{\alpha\to\alpha-1,\beta\to\beta-2,\zeta\to\zeta+2}
 \,.
\end{align}
The last two relations imply
\begin{align}
    &(X^\top-\lambda_n Z^\top)\,\big|_{\alpha\to\alpha-1,\beta\to\beta-2,\zeta\to\zeta+2}\  Z^\top \ket{d_n^*}=0.
\end{align}
Therefore, $Z^\top \ket{d_n^*}$ is proportional to the vectors that are solutions of the GEVP \eqref{eq:GEVPds} when the parameters are shifted according to $\alpha\to\alpha-1,\beta\to\beta-2,\zeta\to\zeta+2$. The proportionality factor is computed by comparing the coefficients of $\ket{d_n^*}$ and  $Z^\top \ket{d_n^*}$ in front of $\ket{n}$. One obtains:
\begin{align}
 \frac{1}{n-\alpha+1}   Z^\top \ket{d_n^*}=\ket{d_n^*}\big|_{\alpha\to\alpha-1,\beta\to\beta-2,\zeta\to\zeta+2}.
\end{align}
Moreover the action of $V$ is not affected by these shifts and, by consequence, its eigenvectors $\ket{e_n}$ are unchanged. One deduces that 
\begin{align}
    \bra{e_m}|Z^\top \ket{d_n^*}=(n-\alpha+1)U_m(n;\alpha-1,\beta-2,\zeta+2).
\end{align}
Using the actions of $Z^\top$ on $\ket{d_n^*}$ computed in subsection \ref{ssec:dds}, one deduces that
\begin{align}
 (n-\alpha+1)U_m(n;\alpha-1,\beta-2,\zeta+2)
=(n-\alpha) U_{m}(n;\alpha,\beta,\zeta) + \dfrac{n-2\alpha+\beta}{n-\alpha}U_{m}(n-1;\alpha,\beta,\zeta).
\end{align}
Inserting the expression \eqref{eq:Un} of $U_m(n)$ in terms of $\mathcal{U}_m(n)$, one proves the result of the proposition.
\end{proof}
\begin{remark}
The function $\mathcal{U}_m(n)$ defined in \eqref{eq:rational-cU} can be obtained by setting $s=q^{N-a}$, $\beta=q^{2a+b+2c-2N}$ and $\delta=q^{N-2b-2c-1}$ and then taking the limit $q \to 1$ of the function
\begin{equation}\label{eq:R1}
{\mathscr{R}}^{(1)}_m(n)={}_{4}\phi_3 \left({{ q^{-m},\;\delta q^{m-N},\;q^{-n},\;sq^{-N} }\atop
{q^{-N} ,\; q^{-N}/\beta,\; \beta \delta sq^{1-n}}}\;\Bigg\vert \; q;q\right)\,,
\end{equation}
introduced in \cite{GM}. 
Moreover, in the same article the function \eqref{eq:R1} itself is shown to be obtained as a limit of the Wilson rational functions.
\end{remark}
\section{Differential relization of $m\sR$ \label{sec:models}}
This section presents an explicit differential realization of the meta Racah algebra.
Consider the following differential operators (we use the same symbols for the abstract generators
and their realizations):
\begin{align}
Z &= x(1-x)\frac{d}{dx}+(Nx-\alpha) I,\\
V &= x(1-x)\frac{d^2}{dx^2}+\left(2x(\beta+\zeta)  + N-2\alpha-\beta- 2\zeta\right)\frac{d}{dx} -(\beta+\zeta)(\beta+\zeta+1 ) I,\\
X &= -x^2(1-x)\frac{d^2}{dx^2} -x((N+\beta-1)x-2\alpha+1)\frac{d}{dx} + (N \beta x-\alpha^2) I.
\end{align}
A direct computation shows that these operators realize the defining relations \eqref{mRA1}–\eqref{mRA3} of the meta Racah algebra, with the parameters $\xi$ and $\eta$ given by \eqref{eq:xi2} and \eqref{eq:eta2}. As already observed in \cite{GIVZTridiag,LabrietPoulain}, the operators $Z$ and $V$ generate a Hahn algebra (see Section~\ref{sec:metaR}). We take these operators as acting on $\sV_N$, the space of polynomials of degree at most $N$, and consider the monomial basis $g_n(x)=(-1)^n(-N)_n x^n$, $n=0,\ldots, N.$  The action of the operators on this basis reads
\begin{align}
     {Z}g_n(x) &= (n-\alpha) g_n(x) +  g_{n+1}(x)\,, \qquad {Z} g_N(x)= (N-\alpha) g_N(x)\,, \\
     {V}g_n(x) &= (n-\beta-\zeta-1)(\beta+\zeta-n)g_n(x)+n(N+1-n)(n-1-2\alpha-\beta-2\zeta+N)g_{n-1}(x)\,,\\
     {X}g_n(x) &= -(n-{\alpha})^2 g_n(x) - (n-{\beta}) g_{n+1}(x)\,,\qquad Vg_N(x)=-(N-\alpha)^2 g_N(x)\,.
\end{align} 
Comparing with the definitions \eqref{eq:Zf}-\eqref{eq:Xf} of the abstract representation, we conclude that the functions $g_n$ model the vectors $\ket{n}$ of the standard basis.

We realize the bilinear pairing $\langle f, g\rangle$ by the contour integral  
\begin{equation}
    \langle f, g \rangle=\frac{1}{2\pi i}\oint_\Gamma f(x) g(x) dx\,,
\end{equation}
where $\Gamma$ is the circle $|x|<a<1$. Notice that this is not an inner product on $\sV_N$ but this pairing identifies $(\sV_N)^*$ with the vector space generated by $\,\{\,x^{-n-1}:n=0,\ldots,N\,\}$.
Defining $g_n^*(x)=(-1)^{n}x^{-n-1}/(-N)_{n}$ for  $n=0,\ldots,N$ one finds 
$\langle g_m^*, g_n\rangle=\delta_{n,m}\,.$ Thus the functions $g_n^*$ realize the vectors $\bra{n}\,|$ in this identification. We use this pairing to compute the transpose operators for $Z,\ X,\ V$ and one can then check that on the basis $g_n^*$ they are given by the formulas $\eqref{eq:ZT-standard}-\eqref{eq:XT-standard}$ which read
\begin{align}
Z^\top g_n^*(x)&=(n-\alpha)g_n^*(x)+g_{n-1}^*(x)\,,\quad Z^\top g_0^*(x)=(-\alpha)g_0^*(x)\,,\\
 V^{\top}g_n^*(x) &= (n-\beta-\zeta-1)(\beta+\zeta-n)g_n^*(x)+(n+1)(N-n)(n-2\alpha-\beta-2\zeta+N)g_{n+1}^*(x)\,,\\
          X^{\top}g_n^*(x) &= -(n-\alpha)^2 g_n^*(x) - (n-1-\beta) g_{n-1}^*(x)\,,\quad X^\top g_0^*(x)=-\alpha^2g_0^*(x)\,.
\end{align}
Integration by parts yields the following differential expressions for $Z^\top$, $V^\top$, and $X^\top$
\begin{align}
Z^\top &= -x(1-x)\frac{d}{dx}+(x(N+2)-\alpha-1) I,\\
V^\top &= x(1-x)\frac{d^2}{dx^2}-(2x(\beta+\zeta+2)+N-2\alpha-\beta-2\zeta-2
)\frac{d}{dx}-(\beta+\zeta+1)(\beta+\zeta+2)I\,,\\
X^\top &=-x^2(1-x) \frac{d^2}{dx^2} +x\left(\left(N +\beta +5\right) x -2 \alpha -3\right)
\frac{d}{dx} +((\beta+2)(N+2)x-(\alpha+1)^2) I.
\end{align}

\begin{remark}\label{rem:QuotientSpace}
    Notice that these differential operators  coincide with the action on the basis $g_n^*$ except for the action of $V^\top$ on $g_N^*$ and the actions of $X^\top,\ Z^\top$ on $g_0^*$. This is solved as follows. The differential realization for $X,\ Z,\ V$ can be seen to act on the space of Laurent series, and the subspace $\sV_N$ is then an irreducible subrepresentation. This implies that the transpose operators will act on a quotient of the space of  Laurent series which will be isomorphic to $(\sV_N)^*$. In other words one has to impose $g_{N+1}^*=g_{-1}^*=0$.
\end{remark}
Given the above realization, explicit models for the bases $d, d^*, e, e^*, f, f^*, z$ and $z^*$ can be constructed. These bases can be constructed either by solving the associated differential equations or by using the results of Section \ref{sec:eigenbases}. 

We begin by observing that the operator $V$ coincides with the Jacobi differential operator on the interval $(0,1)$. Consequently, in this model the basis $e$ is expressed in terms of Jacobi polynomials with parameters $ a=N-2\alpha-\beta-2\zeta-1$ and $b=2\alpha-\beta-N-1$:
\begin{align}
e_n(x)
&=\frac{(-N,N-2\alpha-\beta-2\zeta)_n}{(n-2\beta-2\zeta-1)_n} \hg{2}{1}\argu{-n,n-2\beta-2\zeta-1}{N-2\alpha-\beta-2\zeta}{x}\\
&=\frac{n!(-N)_n}{(n-2\beta-2\zeta-1)_n} J_n^{(N-2\alpha-\beta-2\zeta-1,2\alpha-\beta-N-1)}(x)\,. 
\end{align}
Recall the definition of the Jacobi polynomials on $(0,1):$
\begin{equation}
J_n^{(a,b)}(x)=\frac{(a+1)_n}{n!}\hg{2}{1}\argu{-n,n+a+b+1}{a+1}{x}\,.
\end{equation}
The bases $d$, $f$ and $z$ are given by 
 \begin{align}
 d_n(x)
 &=(-1)^n(-N)_nx^n\hg{2}{1}\left({{n-N,\alpha-\beta}\atop
{-\alpha+n+1}}\;\Bigg\vert \; x\right)\,,\\
   f_n(x) 
   &=(-1)^n(-N)_n x^n\hg{2}{1}\argu{n-N, n-\beta-\rho}{2n-2\alpha-\rho+1}{x}\,,\\
   z_n(x)&=(-1)^n(-N)_n x^{n}\left(1-x \right)^{N -n}\,.
   \end{align}
 The action of the operator $Z$ on $d_n$ is given by
\begin{align*}
    Zd_n(x) &=(n-\alpha)(-1)^n(-N)_n x^n \hg{2}{1}\argu{-N+n, \alpha-\beta-1}{-\alpha+n}{x}\,.
\end{align*}
The dual bases on the span of $ \{\,x^{-n-1}:n=0,\ldots,N\,\}$ are given by 
\begin{align}
     d_n^*(x) &=x^{-n-1}\frac{(-1)^{n+1}}{(-N)_n}\sum_{\ell=0}^n \frac{(\beta-\alpha+1,1+N-n)_\ell}{\ell! (\alpha-n)_{\ell+1}} x^\ell\,.\\
 e_n^*(x)  &=x^{-n-1}\frac{(-1)^n}{(-N)_n}\sum_{\ell=0}^{N-n}\frac{(n+1,N+n-2\alpha-\beta-2\zeta)_\ell}{\ell!(2n-2\beta-2\zeta)_\ell}\frac{1}{x^\ell}\,.\\
  f_n^*(x)&=x^{-n-1}\frac{(-1)^n}{(-N)_n}\sum_{\ell=0}^n \frac{(\beta+\rho+1-n,1+N-n)_\ell}{\ell!(1+2\alpha+\rho-2n)_\ell}x^\ell\,,\\
  z^*_n(x)&=x^{-n-1}\frac{(-1)^n}{(-N)_n}\sum_{\ell=0}^n \frac{(-1)^l(1+N-n)_l}{l!}x^\ell \,.
    \end{align}
\begin{remark} 
In the spirit of Remark \ref{rem:QuotientSpace} these eigenvectors coincide with the following function on the quotient space of Laurent series
\begin{align}
    d_n^*(x)
    &\equiv \frac{(-1)^n}{(-N)_n(n-\alpha)} x^{-n-1}\hg{2}{1}\argu{\beta-\alpha+1,1+N-n}{\alpha-n+1}{x}\,, \\
    e_n^*(x) 
    &\equiv \frac{(-1)^n}{(-N)_n}x^{-n-1}\hg{2}{1}\argu{n+1, N+n-2\alpha-\beta-2\zeta}{2n-2\beta-2\zeta}{\frac{1}{x}}\,,\\
    f_n^*(x)
    &\equiv \frac{(-1)^n}{(-N)_n}x^{-n-1} \hg{2}{1}\argu{\beta+\rho+1-n,1+N-n}{1+2\alpha+\rho-2n}{x}\,,\\
    z^*_n(x)&\equiv \frac{(-1)^n}{(-N)_n} x^{-n -1}(1-x)^{n -1-N}\,.
\end{align}
An important remark for our purpose is that if $f\equiv g$ then $\langle f, h\rangle=\langle g, h\rangle$ for all $h\in \sV_N$. This will be used in the proof of the following Proposition.
\end{remark}
    
The following proposition can be obtained as a consequence of the orthogonality relations of Section \ref{sec:eigenbases} but we recover them here using residue computations. 
\begin{pr}
\label{eq:orth-model1}
The following orthogonality relations hold: 
\begin{align}
    &\langle f_m^*, f_n\rangle =\langle f_n, f_m^*\rangle =\delta_{n,m}\,,\\
     &\langle e_m^*, e_n\rangle =\langle e_n, e_m^*\rangle =\delta_{n,m}\,,\\
     &\langle z_m^*,z_n\rangle=\langle z_n,z_m^*\rangle=\delta_{n,m}\,,\\
      & \langle d_m^*,Z d_n \rangle =\langle  d_m, Z^\top d_n^*\rangle = \delta_{n,m}\,.
\end{align}
\end{pr}
\begin{proof}
We begin by establishing the relation $\langle f_m^*, f_n\rangle =\delta_{n,m}$. To this end, introduce the function
\begin{align}
h_{m,n}(x)=\hg{2}{1}\argu{1+\rho+\beta-m,1+N-m}{1+2\alpha+\rho-2m}{x}\hg{2}{1}\argu{n-N, n-\beta-\rho}{2n-2\alpha-\rho+1}{x}\,.
\end{align}
By definition of $f$ and $f^*$ and by Cauchy theorem we have
\begin{align}
\label{eq:model1-ortf}
\langle f_m^*, f_n\rangle &=\frac{\epsilon_m^*\epsilon_n}{2\pi i}\oint _\Gamma \frac{h_{m,n}(x)}{x^{m-n+1}}dx=\begin{cases}1\,\,\,\hspace{1.9cm}m=n\\
(-1)^{n+m}\frac{(-N)_n}{(-N)_m}\frac{h^{(m-n)}_{m,n}(0)}{(m-n)!}\qquad m>n\\
0\hspace{2.6cm} m<n
\end{cases}
\end{align} 
It therefore remains to show that $h^{(s)}_{n+s,n}(0)=0$ for all $s>0$. Using Leibniz’s rule together with the identity $(a)_s=(-1)^k(a)_{s-k}(1-a-s)_k$
and applying the Chu–Vandermonde identity, we compute 
\begin{align}
h^{(s)}_{n+s,n}(0)&=\sum_{k=0}^s\binom{s}{k}\frac{(1+\rho+\beta-n-s)_{s-k}(1+N-n-s)_{s-k}(n-N)_k(n-\beta-\rho)_k}{(1-2n-2s+2\alpha+\rho)_{s-k}(2n-2\alpha-\rho+1)_k}\\
&=\frac{(1+\rho+\beta-n-s)_s(1+N-n-s)_s}{(1-2n-2s-2\alpha+\rho)_s}\sum_{k=0}^s \frac{(-s,2n+2\alpha-\rho+s)_k}{(2n-2\alpha-\rho+1)_k}\nonumber\\
&=\frac{(1+\rho+\beta-n-s)_s(1+N-n-s)_s}{(1-2n-2s-2\alpha+\rho)_s} \hg{2}{1}\argu{-s,2n-2\alpha-\rho+s}{2n-2\alpha-\rho+1}{1}\nonumber\\
&=\frac{(1+\rho+\beta-n-s)_s(1+N-n-s)_s}{(1-2n-2s-2\alpha+\rho)_s}  \frac{(1-s)_s}{(2n-2\alpha-\rho+1)_s}=0\,\nonumber.
\end{align}
The identities $\langle z_m^*,z_n\rangle=\delta_{n,m}$ and $\langle d_m^*, Z d_n\rangle =\delta_{n,m}$ follow by an analogous argument and are therefore omitted.
The Laurent expansion around $x=0$ of the product $e_m^*(x)e_n(x)$ is
\begin{equation}
e_m^*(x)e_n(x)=\frac{(-N,N-2\alpha-\beta-2\zeta)_n(-1)^m}{(n-2\beta-2\zeta-1)_n(-N)_m}\sum_{j=0}^\infty
\sum_{k=0}^n
A_{n,k}B_{m,j}\,
x^{k-m-1-j}\,,
\end{equation}
with 
\begin{equation}
    A_{n,k}
=
\frac{(-n)_k(n-2\beta-2\zeta-1)_k}
{(N-2\alpha-\beta-2\zeta)_k\,k!}\,,\qquad B_{m,j}
=
\frac{(m+1)_j(N+m-2\alpha-\beta-2\zeta)_j}
{(2m-2\beta-2\zeta)_j\,j!}\,.
\end{equation}
Thus, the only terms that contribute to the integral are those  satisfying $
k-m-1-j=-1$, \textit{i.e.}, $
k=m+j.$ 
If $m>n$, then $m+j>n$ for all $j\ge0$, and hence the condition
$k\le n$ cannot be satisfied.
Therefore no term of order $x^{-1}$ appears in the Laurent expansion, and $\langle e_n^*,e_m\rangle=0$ if $m>n$. In the case $m<n$, the condition $k=m+j\le n$ allows $0\le j\le n-m,$.
The coefficient of $x^{-1}$ is then given by
\begin{align}
\frac{(-N,N-2\alpha-\beta-2\zeta)_n(-1)^m}{(n-2\beta-2\zeta-1)_n(-N)_m}\sum_{j=0}^{n-m}
\frac{(-n)_{m+j}(n-2\beta-2\zeta-1)_{m+j}}
{(N-2\alpha-\beta-2\zeta)_{m+j}(m+j)!}
\,
\frac{(m+1)_j(N+m-2\alpha-\beta-2\zeta)_j}
{(2m-2\beta-2\zeta)_j\,j!}.
\end{align}
Using the identities $(a)_{m+j}=(a)_n\,(a+m)_j,$  $(k+m)!=(m+1)_k m!$, and the Chu-Vandermonde identity, 
the sum becomes 
\begin{align}
&\frac{(-N,N-2\alpha-\beta-2\zeta)_n(-1)^m}{(n-2\beta-2\zeta-1)_n(-N)_m}\frac{(-n)_m(n-2\beta-2\zeta-1)_m}{(N-2\alpha-\beta-2\zeta)_m m!}\hg{2}{1}\argu{m-n,n+m-2\beta-2\zeta-1}{2m-2\beta-2\zeta}{1}\nonumber\\
&=\frac{(-N,N-2\alpha-\beta-2\zeta)_n(-1)^m}{(n-2\beta-2\zeta-1)_n(-N)_m}\frac{(-n)_m(n-2\beta-2\zeta-1)_m}{(N-2\alpha-\beta-2\zeta)_m m!} \frac{(1-n+m)_{n-m}}{(2m-2\beta-2\zeta)_{n-m}}=0\,,
\end{align}
and we conclude $\langle e_n^*,e_m \rangle=0 $ if $m<n.$
In the case $n=m$, the only term contributing to the term of order $x^{-1}$ is $j=0$, $k=m$, and the coefficient of $x^{-1}$ is equal to $1.$ We conclude $\langle e_m^*,e_n\rangle=\delta_{n,m}$.
\end{proof}
We end this section with integral representations of the Racah polynomials and the rational functions of Racah type. 
\begin{cor}
The Racah polynomials $S_m(n)=\langle f_n^*,e_m\rangle$, the rational functions of Racah type 
$U_m(n)=\langle e_m,d_n^*\rangle$, and the dual Hahn polynomials
\begin{equation}
    R_k^{(dH)}(m;\boldsymbol\rho)
=\frac{k!(m-2\beta-2\zeta-1)_m}{(-N,N-2\alpha-\beta-2\zeta)_m}
\langle e_m,z_k^*\rangle\,,
\end{equation}
admit the following contour integral representations:
\begin{align}
&S_m(n)=\frac{(-1)^n}{2\pi i}\frac{m!(-N)_m}{(-N)_n(m-2\beta-2\zeta-1)_m}\oint _\Gamma x^{-n-1}J_m^{(a,b)}(x)\hg{2}{1}\argu{1+\beta+\rho-n, 1+N-n}{1+2\alpha+\rho-2n}{x} dx\,,\\
&U_m(n)=\frac{(-1)^n}{2\pi i}\frac{m!(-N)_m}{(-N)_n(n-\alpha)(m-2\beta-2\zeta-1)_m}\oint _\Gamma   x^{-n-1}J_m^{(a,b)}(x)\hg{2}{1}\argu{N+1-n,\beta-\alpha+1}{\alpha-n+1}{x}dx\,,\\
&R_k^{(dH)}(m;\boldsymbol\rho)=\frac{(-1)^k}{2\pi i}\frac{m! k!}{(N-2\alpha-\beta-2\zeta)_m(-N)_k}\oint_{\Gamma} x^{-k-1}(1-x)^{k-1-N}J_m^{(a,b)}(x)dx\,.
\end{align}
Here $J_m^{(a,b)}(x)$ denotes the Jacobi polynomial with parameters $ a=N-2\alpha-\beta-2\zeta-1$ and $b=2\alpha-\beta-N-1$ and $\boldsymbol\rho=(N-2\alpha-\beta-2\zeta-1,2\alpha-\beta-N-1,N)$.
\end{cor}
\section{Conclusion}
The program of extending the Askey scheme to biorthogonal rational functions through the introduction of meta algebras and their representations, initiated in \cite{tsujimoto2024meta} with the Hahn case, and continued in \cite{bernard2024meta} with the $q$-Hahn case, was pursued in the present paper by dealing with the Racah case. Following the general approach of \cite{tsujimoto2024meta}, the meta Racah algebra with three generators $X,V, Z$
 was defined abstractly, and a representation on a finite-dimensional vector space where all three generators act in a bidiagonal fashion was obtained. This bidiagonal representation allowed us to explicitly solve the relevant generalized and ordinary eigenvalue problems involving the operators 
$X,V,Z$ or their transpositions. The representations of the meta Racah algebra in these eigenbases were characterized. The overlap coefficients between solutions of different ordinary eigenvalue problems were identified with Racah polynomials, while those involving generalized eigenvalue problems gave rise to Racah type rational functions. Their orthogonality, biorthogonality, and bispectral properties were shown to follow directly from the algebraic relations and the representation theory of the meta Racah algebra.
It was found that the operators $V$ and $Z$ form the Hahn algebra, and that 
$X$ is an algebraic Heun operator associated with the Leonard pair $(V,Z).$
Furthermore, the construction developed was seen to relate to the Leonard trios formalism. All this extends therefore the correspondence between Leonard pairs and the terminating families in the Askey scheme to one that includes biorthogonal rational functions.

These results open several directions for further research. A natural next step is the study of infinite dimensional representations of meta algebras, which are expected to provide an algebraic framework for infinite families of orthogonal polynomials and biorthogonal rational functions. Another promising direction concerns multivariate generalizations, where meta algebras may offer a systematic approach to the construction of multivariate bispectral rational functions. Finally, it would be particularly interesting to apply the meta algebraic approach to the Bannai–Ito scheme \cite{genest2014superintegrability,genest2014racah,tsujimoto2012dunkl,genest2016non,genest2014bannai,de2015bannai,tsujimoto2013dual}. Given the role of reflection operators and symmetries in this setting, suitable meta extensions of the corresponding algebras could provide new algebraic interpretation of the Bannai–Ito polynomials and their rational analogues, further illuminating their position within a broader bispectral scope.

Further directions include the study of meta algebras associated with Wilson rational functions and elliptic biorthogonal rational functions; in particular, an open problem is to clarify the relation between meta algebras of Askey–Wilson type and elliptic Sklyanin algebras, given that the most general explicitly known elliptic biorthogonal rational functions are related to elliptic quadratic Sklyanin algebras, as observed by Rains and Rosengren \cite{Rosengren}.
\vspace{1cm}
\paragraph{\textbf{Acknowledgements:}}
The research of ST is supported by JSPS KAKENHI (Grant Number 24K00528). 
L.~Vinet is funded in part by a Discovery Grant from the Natural Sciences and Engineering Research Council (NSERC) of Canada. Q.~Labriet and L.~Morey enjoy postdoctoral fellowships provided by this grant. The work of A.~Zhedanov was performed at the Saint Petersburg Leonhard Euler International Mathematical Institute and supported by the Ministry of Science and Higher Education of the Russian Federation (agreement no. 075–15–2025–343).  

\bibliographystyle{utphys}
\bibliography{metaracah.bib}

\end{document}